\documentclass[12pt]{amsart}
\usepackage{bbm}
\usepackage{amscd, amssymb}
\input{xy}
\xyoption{all}

\newtheorem{Thm}{Theorem}[subsection]
\newtheorem{Conj}[Thm]{Conjecture}
\newtheorem{Prop}[Thm]{Proposition}
\newtheorem{Def}[Thm]{Definition}
\newtheorem{Def/Thm}[Thm]{Definition/Theorem}
\newtheorem{Cor}[Thm]{Corollary}
\newtheorem{Lemma}[Thm]{Lemma}

\theoremstyle{remark}
\newtheorem{Rmk}[Thm]{Remark}

\newtheorem{EG}[Thm]{Example}
\numberwithin{equation}{subsection}

\newcommand{\ot }{\otimes}
\newcommand{\ra }{\rightarrow}

\newcommand{\lra }{\longrightarrow}
\newcommand{\Map}{{\mathrm{Map}}}

\newcommand{\Proj}{{\mathrm{Proj}}}

\newcommand{\Hom }{{\mathrm{Hom}}}

\newcommand{\Spec}{{\mathrm{Spec}}}

\newcommand{\Pic}{{\mathrm{Pic}}}

\newcommand{\cO}{{\mathcal{O}}}

\newcommand{\cL}{{\mathcal{L}}}
\newcommand{\cE}{{\mathcal{E}}}

\newcommand{\cP}{{\mathcal{P}}}

\newcommand{\cC}{{\mathcal{C}}}

\newcommand{\fC}{{\mathfrak{C}}}
\newcommand{\fP}{{\mathfrak{P}}}
\newcommand{\G}{{\bf G}}
\newcommand{\bH}{{\bf H}}

\newcommand{\bS}{{\bf S}}

\newcommand{\NN}{{\mathbb N}}

\newcommand{\PP }{{\mathbb P}}

\newcommand{\QQ }{{\mathbb Q}}
\newcommand{\CC }{{\mathbb C}}
\newcommand{\ZZ }{{\mathbb Z}}
\newcommand{\RR }{{\mathbb R}}

\newcommand{\one }{{\mathbbm 1}}

\newcommand{\ke }{{\varepsilon }}

\newcommand{\ka }{{\alpha}}

\newcommand{\Mgk}{\overline{M}_{g,k}}

\newcommand{\fMgk}{\mathfrak{M}_{g,k}}

\newcommand{\Mlong}{\overline{M}_{0,k+\bullet}(\WmodG,\beta)}

\newcommand{\Qlong}{\mathrm{Q}_{0,k+\bullet}(\WmodG,\beta)}

\newcommand{\Qglonge}{\mathrm{Q}^\ke_{g,k+\bullet}(\WmodG,\beta)}

\newcommand{\Qlonge}{\mathrm{Q}^\ke_{0,k+\bullet}(\WmodG,\beta)}

\newcommand{\Qgonelonge}{\mathrm{Q}^\ke_{g_1,k_1+\bullet}(\WmodG,\beta_1)}
\newcommand{\Qgtwolonge}{\mathrm{Q}^\ke_{g_2,k_2+\bullet}(\WmodG,\beta_2)}

\newcommand{\QGge}{QG^\ke_{g,k,\beta}(\WmodG)}

\newcommand{\QGraphe}{QG^\ke_{0,k,\beta}(\WmodG)}
\newcommand{\eb}{ev_\bullet}

\newcommand{\WmodG}{W/\!\!/\G}

\newcommand{\QmapW}{\mathrm{Q}_{g,k}(\WmodG,\beta)}
\newcommand{\QmapWe}{\mathrm{Q}_{g,k}^\ke(\WmodG,\beta)}

\newcommand{\BunG}{\mathfrak{B}un_\G}

\newcommand{\T}{\bf T}
\newcommand{\fM}{\mathfrak{M}}

\newcommand{\lan}{\langle}
\newcommand{\ran}{\rangle}
\newcommand{\lla}{\langle\!\langle}
\newcommand{\rra}{\rangle\!\rangle}

\begin{document}
\title[Wall-crossing in $g=0$ quasimap theory and mirror maps]{Wall-crossing in genus zero quasimap theory and mirror maps}

\begin{abstract} For each positive rational number $\ke$, the theory of $\ke$-stable quasimaps to certain GIT quotients $\WmodG$ developed in \cite{CKM} gives rise to 
a Cohomological Field Theory. Furthermore, there is an asymptotic theory corresponding to $\ke\ra0$.
For $\ke >1$ one obtains the usual Gromov-Witten theory of $\WmodG$, while the other theories are new. However, they are all expected to contain the same information and
in particular the numerical invariants should be related by wall-crossing formulas. In this paper we analyze the genus zero picture and find that the wall-crossing in this case significantly
generalizes toric mirror symmetry (the toric cases correspond to abelian groups $\G$). In particular, we give a geometric interpretation of the mirror map as a generating
series of quasimap invariants. We prove our wall-crossing formulas for all targets $\WmodG$ which admit a torus action with isolated fixed points, as well as for zero loci
of sections of homogeneous vector bundles on such $\WmodG$.
\end{abstract}

\author{Ionu\c t Ciocan-Fontanine}
\noindent\address{School of Mathematics, University of Minnesota, 206 Church St. SE,
Minneapolis MN, 55455, and\hfill
\newline \indent School of Mathematics, Korea Institute for Advanced Study,
85 Hoegiro, Dongdaemun-gu, Seoul, 130-722, Korea}
\email{ciocan@math.umn.edu}

\author{Bumsig Kim}
\address{School of Mathematics, Korea Institute for Advanced Study,
85 Hoegiro, Dongdaemun-gu, Seoul, 130-722, Korea}
\email{bumsig@kias.re.kr}\maketitle

\section{Introduction} 
\subsection{Overview} For a complex affine algebraic variety $W$ acted upon by a reductive group $\G$, a choice of a 
character $\theta$ of $\G$ determines a linearized line bundle $L_\theta$ on $W$,
and hence a GIT quotient $\WmodG=W/\!\!/_\theta \G$ with a canonical (relative) polarization $\cO(\theta)$. 
Once such a $\theta$ is fixed, and under reasonable conditions on the triple $(W,\G,\theta)$, 
certain stability conditions (depending on a parameter $\ke\in \QQ_{>0}$ ) are imposed on the moduli stacks
$\fMgk([W/\G],\beta)$ parametrizing maps of (homology) class $\beta$ from $k$-pointed nodal curves of genus $g$ to the quotient {\it stack} $[W/\G]$. This produces
(relatively) proper Deligne-Mumford moduli stacks of $\ke$-{\it stable quasimaps},
carrying virtual fundamental classes. They come equipped with evaluation maps and with tautological cotangent $\psi$-classes at the markings.
Moreover, the same holds for an asymptotic stability condition denoted by $\ke=0+$ and corresponding to $\ke\ra 0$. 

These structures allow the introduction of numerical invariants by integrating natural cohomology classes against the virtual class. 
Explicitly, for $\ke\geq 0+$, let $\QmapWe$ denote the moduli space of
$\ke$-stable quasimaps of class $\beta$. When $\WmodG$ is projective, its (descendant) $\ke$-quasimap invariants are 
defined by
$$\langle \gamma_1\psi_1^{a_1},\dots,\gamma_k\psi_k^{a_k}
\rangle^\ke_{g,k,\beta}:=\int_{[\QmapWe]^{\mathrm{vir}}}\prod_{j=1}^k  ev_j^*(\gamma_j)\psi_j^{a_j},$$
for integers $a_j\geq 0$ and cohomology classes $\gamma_j\in H^*(\WmodG,\QQ)$. 

If $\WmodG$ is only quasiprojective one may still define a good theory of {\it equivariant} quasimap invariants in the presence of an action with 
good properties by an algebraic torus, 
via the virtual localization formula.

For each fixed class $\beta$ the set $\QQ_{>0}$ is divided into stability chambers by finitely many walls $1,\frac{1}{2},\dots,\frac{1}{\beta(L_\theta)}$ such that
the moduli spaces (and therefore the invariants) stay constant in each chamber. Here $\beta(L_\theta)$ is essentially the degree on the curve class
$\beta$ of the polarization $\cO(\theta)$. For $\ke\in(1,\infty)$ one obtains  the Gromov-Witten invariants of $\WmodG$; we write $\ke=\infty$ for these stability conditions.
On the other hand, the first stability chamber $(0,\frac{1}{\beta(L_\theta)}]$ depends on $\beta$ and the asymptotic stability condition $\ke=0+$ corresponds to being
in this first chamber for {\it all} $\beta$.

The foundations of stable quasimaps were established in \cite{CKM}, as a unified perspective on earlier constructions in \cite{MOP}, \cite{CK}, 
\cite{MM}, \cite{Toda} in some special cases; a summary is included in \S 2 below. 
The theory applies to a large class of targets, including  a host of varieties whose Gromov-Witten theory has received attention
over the last two decades: toric and flag varieties, as well as zero loci of sections of homogeneous bundles on them, local targets with base a GIT quotient, 
Nakajima quiver varieties etc. See \S\ref{Examples} for a list and more details.

 As explained in \cite{CKM}, one has in fact  a Cohomological Field Theory in the sense of Kontsevich and Manin for each value
of $\ke$, and the expectation is that these CohFT's obtained by varying a stability condition are equivalent. In particular, 
one should express the correspondence
between the numerical invariants via wall-crossing formulas. The goal of this paper is to formulate such wall-crossing formulas for the genus zero sector of the
theory. What we discover is that they  turn out to be a significant generalization of Givental's mirror
theorems for semi-positive complete intersections in toric varieties. 
At the same time, we give a geometric interpretation to the {\it mirror map} as a generating
series for certain quasimap invariants.

\subsection{Wall-crossing} The genus zero wall-crossing formulas are most naturally expressed via generating functions for quasimap invariants.

Let the {Novikov ring} $\Lambda=\QQ[[q]]$
be the $q$-adic completion of the semigroup ring on the $L_\theta$-effective curve classes (see Definition \ref{t-effective} for the notion of $L_\theta$-effective classes).
Let $\{\gamma_i\}$ be a homogeneous
basis of the rational cohomology $H^*(\WmodG)$ (equivariant, in the non-compact case) and let $\{ \gamma^i\}$ be the dual basis 
with respect to the intersection pairing. Denote by ${\bf t}=\sum_it_i\gamma_i$ the general element and
by $\one$ the unit in
$H^*(\WmodG)$.

Let $z$ be  a formal variable. For $0+\leq \ke$ and $\gamma\in H^*(\WmodG,\Lambda)$, put

$$S^\ke_{\bf t}(\gamma):=\gamma+\sum_i\gamma_i\sum_{(m,\beta)\neq (0, 0)}\frac{q^\beta}{m!} 
\lan \frac{\gamma^i}{z-\psi},\gamma,{\bf t},\dots,{\bf t}\ran_{0,2+m,\beta}^\ke .$$
Our wall-crossing formula (see Conjecture \ref{conj S}, Theorem \ref{equiv Thm1} and Corollary \ref{nonequiv limit}) relates these series for arbitrary $\ke_1,\ke_2$ in different stability chambers (not necessarily adjacent) and invertible classes 
$\gamma$ of the form $\gamma=\one+O(q)$. The case of main interest, when $\gamma=\one$ and we consider wall-crossing from
$\ke$-stable quasimap invariants for an arbitrary $0+\leq\ke\leq 1$ to Gromov-Witten invariants  
is stated here for simplicity.  Let
\begin{equation}\label{string transf intro}\tau^{\ke} ({\bf t}):=
{\bf t}+\sum_i\gamma_i\sum_{\beta\neq 0,m\geq 0}
\frac{q^\beta}{m!}\lan\gamma^i,\one,{\bf t},\dots ,{\bf t}\ran^\ke_{0,2+m,\beta} .\end{equation} 
We then make the following:

\begin{Conj}\label{mirror big intro} For all targets $\WmodG$ and all $\ke \geq 0+$,
$$S^{\infty}_{\tau^{\ke} ({\bf t})}(\one)=S^\ke_{{\bf t}}(\one).
$$
\end{Conj}

In $\S 7$ the following Theorem is proved:

\begin{Thm} \label{Main}
Assume that $W$ admits an action by a torus $\T$, commuting with the action of $\G$ and such that the fixed points of the induced $\T$-action on $\WmodG$ are
isolated. Then Conjecture \ref{mirror big intro} holds. Moreover, if $E$ is a convex $\G$-representation, then the conjecture also holds for the $E$-twisted $\ke$-quasimap
theories of $\WmodG$.
\end{Thm}

The Theorem covers the genus zero quasimap theories of almost all interesting examples of GIT targets listed in $\S\ref{Examples}$: toric varieties, flag varieties, and local 
targets over them all admit torus actions as above. Among Nakajima quiver varieties, interesting examples are the Hilbert schemes of points on $\CC^2$ and on 
resolutions of $A_n$-singularities, as well as their higher rank analogues (note that we do {\it not} require in Theorem \ref{Main} that the $1$-dimensional $\T$-orbits in $\WmodG$ are also isolated).
Finally, zero loci of sections of globally generated homogeneous vector bundles on all the above targets are also covered by identifying their quasimap theories with twisted
theories of the ambient $\WmodG$ - see \cite{CKM} and $\S\ref{twisting}$ below for more details about this identification.

The strategy for proving Theorem \ref{Main} uses virtual localization and is very much inspired by Givental's original approach to the proofs of toric mirror theorems in
\cite{Givental-equiv}, \cite{Givental}, \cite{Givental-elliptic}.
The crucial observation that allows us to extend those arguments is that certain edge factors in localization formulas are $\ke$-independent.

\subsection{$J$-functions and Birkhoff factorization}
The series $S^\infty_{\bf t}(\one)$ in the Gromov-Witten case coincides with Givental's big $J$-function of $\WmodG$ by the string equation. In quasimap theory, 
there are analogous big $J^\ke$-functions for each stability parameter, see \cite{CKM} and $\S 5$ of this paper. In particular, there
is a {\it big} $I$-{\it function} corresponding to the $\ke=0+$ stability condition.
For $\ke\leq 1$ 
these $J^\ke$-functions may contain both positive and negative powers of $z$ and their relation
with the $S^\ke$-operators is given by a more subtle ``Birkhoff factorization" formula. More precisely, we prove (see $\S\ref{Birkhoff}$, 
Theorem \ref{Birkhoff Thm}) the following

\begin{Thm}\label{Birkhoff-intro} For all GIT targets $\WmodG$ with well-defined quasimap theory and
 for every stability parameter $\ke\geq 0+$ there is a naturally defined series 
$P^\ke=P^\ke ({\bf t}, q, z)\in H^*(\WmodG,\Lambda)[[z]]$, convergent in the $q$-adic topology and satisfying
$$J^\ke=S^\ke_{\bf t}(P^\ke).$$
\end{Thm}

The series $P^\ke$ (see \eqref{P-series def}) is a generating series of equivariant virtual intersection numbers on the {\it quasimap graph spaces} $\QGraphe$. 
These are moduli spaces of $\ke$-stable quasimaps whose domain curve has a component which is a parametrized $\PP^1$. The standard $\CC^*$-action on $\PP^1$ lifts 
to the graph spaces and the variable $z$ appears naturally as the equivariant parameter. As the $J^\ke$-functions themselves are defined via 
$\CC^*$-equivariant residues on graph spaces, the Birkhoff factorization formula is obtained from virtual localization for this $\CC^*$-action.

\subsection{Semi-positive targets }  Birkhoff factorization simplifies drastically for {\it semi-positive} targets.
The triple $(W,\G,\theta)$ is called semi-positive if $\beta(\det(T_W))\geq 0$ for all $L_\theta$-effective classes $\beta$.
Here $T_W$ is the $\G$-equivariant virtual tangent bundle of the lci variety $W$.
%the first Chern class of the quotient stack $c_1(T_{[W/\G]})$ is non-negative
%on the semigroup of $L_\theta$-effective classes (this 
Semi-positivity implies that the GIT quotient $\WmodG$ has nef anti-canonical class.

Elementary dimension counting arguments in \S\ref{semi-positive} show  
that the $J$-functions of semi-positive targets have $z$-expansion of the form
$$J^\ke(q,{\bf t},z)=J^\ke_0(q)\one+({\bf t}+J^\ke_1(q))\frac{1}{z}+O(1/z^2)$$
with $J^\ke_0(q)\in\Lambda$ invertible and $J^\ke_1(q)\in H^{\le 2}(\WmodG,\Lambda)$, and that
$$P^\ke(q,{\bf t},z)= J^\ke_0(q)\one.$$ 

Hence Theorem \ref{Birkhoff-intro} has the following
\begin{Cor}\label{semipositive-intro }For semi-positive $(W,\G,\theta)$ and for all $\ke\geq 0+$,
$$\frac{J^\ke(q,{\bf t},z)}{J^\ke_0(q)}=S^\ke_{\bf t}(\one).$$
\end{Cor}

A very special case of the above Corollary is the main result of \cite{CZ}.
Equating coefficients of $1/z$ gives an expression for the transformation \eqref{string transf intro} in terms of $J_0^\ke$ and $J_1^\ke$.
\begin{Cor} For semi-positive $(W,\G,\theta)$ and for all $\ke\geq 0+$,
$$\tau^{\ke} ({\bf t})=\frac{{\bf t}+J^\ke_1(q)}{J^\ke_0(q)}
$$
\end{Cor}

\subsection{Mirror symmetry from wall-crossing} For semi-positive targets the wall-crossing formula translates into the following statement.
\begin{Cor} \label{semi-positive mirror} Assume that Conjecture \ref{mirror big intro} holds for the semi-positive triple $(W,\G,\theta)$. Then for all $\ke\geq 0+$,
$$J^\infty(q,\tau^{\ke}({\bf t}),z)=\frac{J^\ke(q,{\bf t},z)}{J^\ke_0(q)}$$
with $\tau^{\ke} ({\bf t})=\frac{{\bf t}+J^\ke_1(q)}{J^\ke_0(q)}$. 
In particular, this is
true in all cases covered by Theorem \ref{Main}.
\end{Cor}

Directly from definitions, the coefficients $J^\ke_0(q)$ and $J^\ke_1(q)$ for $\ke >0$
depend only on the first stability chamber $(0,\frac{1}{\beta(L_\theta)}]$ and are therefore truncations of the
corresponding series $I_0$ and $I_1$ for the $0+$ stability condition. 
These are the first two terms in the $1/z$-expansion of the {\it small} $I$-function of $\WmodG$, obtained by putting
${\bf t}=0$ in the big $I$-function
\begin{equation}
\label{small I} I^{\mathrm{small}}(q,z):=I(q,0,z)=I_0(q)+\frac{I_1(q)}{z}+O(1/z^2).
\end{equation}
The small $I$-function for a toric variety, or for a complete intersection in a toric variety was introduced by Givental (it differs from \eqref{small I} by an exponential factor); 
quasimap theory extends the notion to general GIT targets $\WmodG$. Geometrically, it is obtained by summing over all classes $\beta$ the $\CC^*$-equivariant (virtual) residues
of certain distinguished fixed-point components in the unpointed graph spaces $QG^{0+}_{0,0,\beta}(\WmodG)$, see \cite{CK}, \cite{CKM}, and \S4-\S5 of this paper. 
These residues are explicitly computable in general and closed hypergeometric-type
formulas for the small $I$-functions are known for almost all interesting examples of targets, such as the ones listed in \S\ref{Examples}.

Givental's Mirror Theorems for semi-positive toric complete intersections \cite{Givental}, and for local toric targets \cite{Givental-elliptic}, relate the {\it small} (Gromov-Witten)
$J$-function to the small 
$I$-function via a change of variables  called the {\it mirror map} on the small parameter space $H^{\leq 2}(\WmodG)$.

Specializing Corollary \ref{semi-positive mirror} to $\ke =0+$ and ${\bf t}=0$ and using the string and divisor equations on the Gromov-Witten side gives precisely this 
type of relation
for all $\WmodG$ (after restoring the omitted exponential factors) and recovers Givental's results in the toric cases.

Furthermore, the quantum corrections to the mirror map are given by 
two-pointed quasimap invariants.
For example, if $\WmodG$ is Calabi-Yau, the classical mirror map 
$$q\mapsto q e^{\int_\beta  \frac{I_1(q)}{I_0(q)}}$$
is determined by
\begin{equation*}\label{mirror map intro}
\frac{I_1(q)}{I_0(q)}=\sum_i\gamma_i\sum_{\beta\neq 0}
q^\beta\lan\gamma^i,\one\ran^{0+}_{0,2,\beta}.
\end{equation*}
with the first sum  only over those $\gamma_i$'s which form a basis of $H^{2}(\WmodG)$. 

More generally, by first taking derivatives $d/dt_{j_1},\dots ,d/dt_{j_n}$ in the equality of Corollary \ref{semi-positive mirror}, then specializing at ${\bf t}=0$ and 
applying the string and divisor equations on the Gromov-Witten side one obtains the following

\begin{Cor}\label{multipoint} Assume that Conjecture \ref{mirror big intro} holds for the semi-positive triple $(W,\G,\theta)$. Then for every stability parameter $\ke\geq 0+$ and every $n\geq 1$
we have 
\begin{align}\label{multipoint formula}e^{\frac{1}{z} \frac{J^\ke_1(q)}{J^\ke_0(q)}}
&\sum_i\gamma_i\sum_\beta q^\beta e^{\int_\beta  \frac{J^\ke_1(q)}{J^\ke_0(q)}} \lan  \frac{\gamma^i}{z-\psi},\gamma_{j_1},\dots, \gamma_{j_n}\ran_{0,1+n,\beta}^\infty =\\
\nonumber &(J_0^\ke(q))^{n-1} \sum_i\gamma_i\sum_\beta q^\beta \lan  \frac{\gamma^i}{z-\psi},\gamma_{j_1},\dots, \gamma_{j_n}\ran_{0,1+n,\beta}^\ke
\end{align}
In particular, \eqref{multipoint formula} is
true in all cases covered by Theorem \ref{Main}. 
\end{Cor}

From this perspective, Conjecture \ref{mirror big intro} and Theorem \ref{Main} generalize toric mirror theorems in several different directions: they
apply to general GIT quotients $\WmodG$ rather than just those with abelian group $\G$, 
there is no semi-positivity restriction, they apply to the {\it big} parameter space $H^*(\WmodG)$, and to all stability conditions $\ke\geq 0+$. In addition,
the transformation \eqref{string transf intro} generalizes the mirror map.

\subsection{Wall-crossing for $J^\ke$-functions} One consequence of Corollary \ref{semi-positive mirror} is that, in the semi-positive cases,
the big $J^\ke$ function lies on Givental's Lagrangian cone $\cL ag_{\WmodG}$
encoding the genus zero Gromov-Witten theory of $\WmodG$. Conjecture \ref{big J} asserts that this remains true for general targets and Theorem \ref{equiv Thm2} proves this 
assertion when there is a torus action with isolated fixed points {\it and} isolated $1$-dimensional orbits. 
We refer the reader to \S\ref{big J wall-crossing} for a detailed formulation.
In this case we do not know an elegant expression similar to \eqref{string transf intro} for the required transformation of coordinates $\tau^{\infty, \ke}({\bf t})$.

\subsection{Further directions} With appropriate modifications, the methods introduced in this paper extend to orbifold GIT quotients and will be 
used by the authors in a forthcoming work to obtain similar
$\ke$-wall-crossing formulas for such targets. In particular, we will obtain orbifold mirror theorems more general than the ones announced in \cite{Iritani} for toric orbifolds
and complete intersections in them.

In a different direction, we expect wall-crossing formulas to hold in higher genus quasimap theory as well and plan to investigate these in future work. 
In particular, we conjecture in \S\ref{higher genus} that the wall-crossing is trivial for Fano triples $(W,\G,\theta)$ of Fano index at least $2$. This is proved in \cite{CKtoric2} for projective toric Fanos.

\subsection{Aknowledgments} Some of the ideas appearing in this work originate from early discussions with Rahul Pandharipande on big 
$I$-functions of toric varieties and from a 
joint project of the authors with Duiliu Emanuel Diaconescu and Davesh Maulik on the relation between the small $I$ and $J$-functions of the Hilbert scheme of points in $\CC^2$. 
It is a pleasure to thank all three of them here.
In addition, special thanks are due to Duiliu Emanuel Diaconescu for many helpful conversations during the preparation of the paper. We thank Jeongseok Oh for pointing out some typos in an earlier version.

The majority of the writing has been done
while the first named author visited KIAS in 2011-2012. He thanks KIAS for financial support, excellent
working conditions, and an inspiring research environment.

The research of I.C.-F. was partially supported by the NSA grant
H98230-11-1-0125 and the NSF grant DMS-1305004. The research of B.K. was partially supported by NRF-2007-0093859.

\section{Stable quasimaps} For the convenience of the reader we recall in this section the basics of the theory developed in \cite{CKM}.
Throughout the paper we will work over the ground field $\CC$.
\subsection{A class of GIT targets}  
%The GIT setup we use can be found in \cite[\S2]{King}. 
Let $W=\Spec(A)$ be an affine algebraic variety with an action by a reductive algebraic group $\G$.
There are two natural quotients associated with
this data: the {\it quotient stack} $[W/\G]$, and the {\it affine quotient} $W/_{\mathrm {aff}}\G=\mathrm{Spec}(A^\G)$.

Denote by $\Pic^\G(W)$ the group of $\G$-linearized line bundles on $W$ and by $\chi(\G)$ the group of characters of $\G$. Each character
 $\theta$ determines a one-dimensional representation $\CC_\theta$ of $\G$, hence a linearized line bundle 
 $$L_\theta:=W\times\CC_\theta\in\Pic^\G(W).$$
 (We will often think of vector bundles in this paper  
 as geometric objects by identifying them with their total space.)
 
Associated to $\theta$ we have the GIT quotient
$$\WmodG:=W/\!\!/_{\theta}\G:=\mathrm{Proj}(\oplus_{n\geq 0}\Gamma(W,L_\theta^{\otimes n})^\G),$$
which is a quasiprojective variety, with a projective morphism
$$\WmodG\lra W/_{\mathrm{aff}}\G$$
to the affine quotient (see \cite[\S2]{King} for the GIT setup used in this paper). 

Let 
$$W^s=W^s(\theta)\;\;\;\;  {\mathrm {and}}\;\;\;\; W^{ss}=W^{ss}(\theta)$$
be the stable (respectively, semistable) open subsets determined by the choice of linearization.
The following will be assumed throughout: 

$(i)$ $\emptyset\neq W^s=W^{ss}.$

$(ii)$ $W^s$ is nonsingular.

$(iii)$ $\G$ acts freely on $W^s$.

It follows that
$\WmodG$ is a nonsingular variety, which coincides with the stack quotient $[W^s/\G]$. In particular,
it is naturally an open substack in $[W/\G]$. The relative polarization over the affine quotient is given by the descended
bundle
$$\underline{L}_\theta:=W^s\times_\G L_\theta.$$
We will also use the notation $\cO(\theta)$ for the polarization. Further, by replacing $\theta$ by a multiple, we may (and will from now on) assume that 
$\cO(\theta)$ is relatively very ample over the affine quotient. 
(Here and in the rest of the paper the above notation is always understood as the usual {\it mixed construction}. Namely, 
given a principal $\G$-bundle $P$ over a base
scheme or algebraic stack $B$, and a scheme U with $\G$-action, we put
$$P\times_\G U:=[(P\times U)/\G].$$
It is an algebraic stack with a representable morphism to $B$ 
$$\rho : P\times_\G U\lra B$$
which is in the {\' e}tale topology a locally trivial fibration
with fiber $U$. )

\subsection{Maps from curves to $[W/\G]$}
Let $(C,x_1,\dots ,x_k)$ be a prestable pointed curve, i.e., a connected projective curve (of some arithmetic genus $g$),
with at most nodes as singularities, together with $k$ distinct and nonsingular marked points on it.
A map $[u]:C\lra [W/\G] $ corresponds to a pair $(P,\tilde{u})$, with 
$$P\lra C$$ 
a principal $\G$-bundle on $C$ and
$$\tilde{u}:P\lra W$$ 
a $\G$-equivariant morphism. Equivalently (and most often), we will consider the data $(P,u)$, with
$$u:C\lra P\times_\G W$$
a section of the fiber bundle $\rho : P\times_\G W\lra C$. Then
$[u]: C \ra [W/\G]$ is obtained as the composite
$$C\stackrel{u}{\ra} P\times _G W \ra [W/\G].$$
Note also that the composition 
$$C\stackrel{[u]}{\ra} [W/\G]\ra W/_{\mathrm {aff}}\G$$
is always a constant map, since $C$ is projective.

\subsection{Numerical class of a map to $[W/\G]$} Recall that $L\mapsto [L/\G]$ gives an identification
$\Pic ^\G(W) = \Pic ([W/\G])$.

Let $(C,P,u)$ be as above.
For  a $\G$-equivariant line bundle $L$ on $W$
we get an induced line bundle $u^*(P\times _\G L) = [u]^*([L/\G])$ on $C$.

By definition, the {\it numerical class} $\beta$ of $(P,u)\in \Map (C, [W/\G])$ is the group homomorphism  
$$\beta:\Pic^\G(W) \ra \ZZ,\;\;\;
 \beta(L) =\deg _C(u^*(P\times _\G L)).$$ 
%To justify the terminology, note that $\beta$ is the image of 
%the class of the $\G$-equivariant cycle
 %$$\begin{CD} P @>\tilde{u}>> W\\
 %@VVV @.\\
%C @. 
 %\end{CD}$$
%under the natural map $H^\G_2(W,\ZZ)\lra \Hom_\ZZ(\Pic^\G(W),\ZZ)$.
%Further, i
If we denote
$$\cL_\theta:=u^*(P\times _\G L_\theta)=P\times_G\CC_\theta$$ and if $[u]$ has image contained in the GIT quotient $\WmodG$, then 
$\cL_\theta=[u]^*\cO(\theta)$, so
$\beta(L_\theta)$ is the usual degree of the map with respect to the polarization $\cO(\theta)$.
              
\subsection{Quasimaps to $\WmodG$ and $\ke$-stability} Fix integers $g,k\geq 0$. We denote by $\fMgk([W/\G],\beta)$ the moduli stack parametrizing
families of tuples
$$((C,x_1,\dots,x_k),P,u)$$
with $(C,x_1,\dots,x_k)$ a prestable curve of genus $g$, $P\lra C$ a principal $\G$-bundle on $C$, and $u$ a section of
$P\times_\G W\lra C$, of degree $\beta$. It contains the Deligne-Mumford stack 
$$M_{g,k}(\WmodG,\beta)$$
of (stable) maps from nonsingular pointed curves to $\WmodG$ as an open substack. However,
it is in general a nonseparated Artin stack of infinite type. More manageable moduli spaces are obtained by imposing
certain {\it stability conditions}.
These will give
compactifications of $M_{g,k}(\WmodG,\beta)$ over $W/_{\mathrm {aff}}\G$, which are
 Deligne-Mumford and carry canonical perfect obstruction theories.

For example, one such stability condition is obtained by requiring first that the image of $[u]$ is contained in $[W^s/\G]=\WmodG$, or,
equivalently that the section $u$ lands inside $P\times_\G W^s$, and second, that the map 
$$[u]:(C,x_1,\dots, x_k)\lra\WmodG$$
is Kontsevich-stable. This of course leads to the familiar moduli stack of stable maps
$\overline{M}_{g,k}(\WmodG,\beta)$.

There are other stability conditions for which
resulting theory is that of {\it $\ke$-stable quasimaps from curves to} $\WmodG$, developed in \cite{CKM} (and inspired by earlier works \cite{MOP}, \cite{CK}, \cite{MM}).
We recall briefly the main points, and refer the reader to {\it loc. cit.} for the details.

\begin{Def}\label{qmap} A quasimap to $\WmodG$ is a map $((C,x_1,\dots,x_k),P,u)$ to the quotient stack
such that for every irreducible component $C_i$ of $C$, with generic point $\xi_i$, we have $u(\xi_i)\in P\times_\G W^s$.
\end{Def}

Hence a quasimap has at most finitely many {\it base points},
i.e., points on $C$ which are mapped into the unstable locus of $W$ by the section $u$.

By \cite[ Lemma 3.2.1]{CKM}, 
if $(C,P,u)$ is a quasimap,
then for any subcurve $C'$ of $C$, the
class $\beta'$ of the restriction of the quasimap to $C'$ satisfies
$$\beta'(L_\theta)\geq 0,\;\; {\mathrm{ with\; equality\; iff}}\;\; \beta'=0.$$
\begin{Def}\label{t-effective}
We say that an element $\beta\in \Hom_\ZZ(\Pic^\G(W),\ZZ)$ is $L_\theta$-{\it effective} if it can be represented as a finite sum of classes of quasimaps. 
\end{Def}
The $L_\theta$-effective classes
form a semigroup, 
denoted $\mathrm {Eff}(W,\G,\theta)$.

\begin{Rmk} 
The "nondegeneracy at generic points" imposed in the definition of a quasimap is clearly an
open condition on the base of families of maps to the stack $[W/\G]$. It is shown in \cite[Theorem 3.2.5 and Remark 3.2.10]{CKM} that quasimaps
to $\WmodG$ with given underlying pointed curve and fixed $\beta(L_\theta)$ form a bounded family. Hence we have an open Artin substack
$$\mathfrak{Qmap}_{g,k}(\WmodG,\beta)\subset \fMgk([W/\G],\beta)$$
which is of finite type over the stack $\fMgk$ of marked prestable curves.
\end{Rmk}

\begin{Def}\label{prestable} A quasimap is {\rm prestable} if its base points are away from the nodes and markings of the underlying curve.
\end{Def}

\begin{Def}\label{length}
The {\em length} $\ell(x)$ at a point $x\in C$ of a prestable quasimap $((C,x_i),P,u)$ to $\WmodG$ is 
 the order of contact of $u(C)$ with the unstable subscheme $P\times_\G W^{us}$ at $u(x)$. Precisely, 
 consider the ideal sheaf $\mathcal{J}$ of the closed subscheme $P\times_\G W^{us}$ of $P\times_\G W$. Then
 \begin{equation}\label{multiplicity}\ell(x)=\mathrm{length}_x(\mathrm{coker}(u^*\mathcal{J}\lra \cO_C)),\end{equation}
 \end{Def}

\begin{Def}\label{stable qmap} Fix $\ke\in\QQ_{>0}$. A quasimap 
$$((C,x_1,\dots,x_k),P,u)$$
is called $\ke$-{\rm stable} if it is prestable and the following two conditions hold 
\begin{enumerate}
\item $\omega _C (\sum_{i=1}^k x_i ) \otimes \mathcal{L}_{\theta} ^\ke$ is ample.
\item  $\ke  \ell(x) \le 1$ for every point $x$ in $C$.
\end{enumerate} 
Here $\cL_\theta=P\times_\G\CC_\theta =u^*(P\times_\G L_\theta)$ .
\end{Def}

It is easy to see that
$\ke$-stable quasimaps form an open moduli substack in $\mathfrak{Qmap}_{g,k}(\WmodG,\beta)$. We will denote it by $\QmapWe$.
The following is proved in \cite[Theorem 7.1.6]{CKM}.

\begin{Thm}\label{proper}
The stack $\QmapWe$ is a Deligne-Mumford stack, separated,  of finite type, with a natural {\rm proper} morphism
$$\QmapWe\lra \mathrm{Spec}(A(W)^\G).$$
\end{Thm}

\begin{Rmk} The description of the points in the moduli spaces when $\ke$ varies follows easily from the definition of $\ke$-stability. We summarize it  below.
\begin{enumerate}
\item Assume $(g,k)\neq(0,0)$. Then a quasimap is a {\it stable map} to $\WmodG$ in the usual sense if 
and only if it is $\ke$-stable for every $\ke >1$. When $(g,k)=(0,0)$ the same holds with $\ke >2$.
Hence we have 
$$\mathrm{Q}^{\infty}_{g,k}(\WmodG,\beta)=\Mgk (\WmodG,\beta).$$
\item Fix $\beta\in \mathrm {Eff}(W,\G,\theta)$. For each $\ke\leq \frac{1}{\beta(L_\theta)}$ the length inequality in the stability condition is always satisfied, while the ampleness
requirement imposes first that the underlying curve $C$ has {\it no rational tails}
(irreducible components which are rational and carry exactly one special point - marking or node - on them), and second that on each {\it rational bridge} 
(irreducible component which is rational and carries exactly two special points) the line bundle $\cL_\theta$ has strictly positive degree. 

Since we may want to deal with all possible degrees $\beta$ at the same time, it is better to reformulate the stability condition in this case as
$$\omega _C (\sum_{i=1}^k x_i ) \otimes \mathcal{L}_{\theta} ^\ke\;\; \text {is ample for \underline{all}}\;\; \ke\in \QQ_{>0}.$$ 
We will call these either $(\ke=0+)$-stable quasimaps, 
or, as in \cite{CKM}, {\it stable quasimaps to $\WmodG$}; the resulting moduli stack will be denoted
$$\mathrm{Q}^{0+}_{g,k}(\WmodG,\beta),$$
or simply
$\QmapW$.
Note that these moduli spaces are empty for $(g,k)=(0,0), (0,1)$.
\item Assume for simplicity that $(g,k)\neq(0,0), (0,1)$. Then for every integer $1\leq e\leq \beta(L_\theta)-1$, the moduli space of $\ke$-stable quasimaps to $\WmodG$
with fixed numerical data $(g,k,\beta)$ stays constant when $\ke\in\left(\frac{1}{e+1},\frac{1}{e}\right]$.
The underlying curve of a quasimap parametrized by this constant moduli space is allowed only
rational tails of total degree at least $e+1$ (with respect to $L_\theta$), while the length at each base
point must be at most ${e}$. Hence for each fixed class $\beta$ the set $\QQ_{>0}$ is divided into chambers by finitely many walls $1,\frac{1}{2},\dots,\frac{1}{\beta(L_\theta)}$.
The two extreme chambers $(1,\infty)$ and $(0,\frac{1}{\beta(L_\theta)}]$ always correspond to stable maps and stable quasimaps, respectively.

\end{enumerate}

\end{Rmk}

\subsection{Obstruction theory}
The moduli stack of $\ke$-stable quasimaps comes with natural forgetful morphisms
$$\mu:\QmapWe\lra\fMgk,\;\;\; \nu:\QmapWe\lra\BunG,$$
where $\fMgk$ is the moduli stack of prestable curves and $\BunG$ is the
moduli stack of principal $\G$-bundles on the fibers of the universal curve
$\pi:\fC_{g,k}\lra\fMgk$.

There is a 
a universal family
$$((\cC^\ke,x_1,\dots, x_k),\cP,u)$$
over $\QmapWe$, where the universal curve
$$\pi:\cC^\ke\lra\QmapWe$$ 
is the pull-back of $\fC_{g,k}$ via $\mu$, $\cP$ is the pull-back of the
universal principal $\G$-bundle on (the universal curve over) $\BunG$ via $\nu$, and
$u$ is a section of the fiber bundle
$$\varrho:\cP\times_G W\lra\cC^\ke$$
(the universal section).

The canonical obstruction theory of $\QmapWe$, relative to the smooth Artin stack $\BunG$, is given by
the complex
\begin{equation}\label{obs theory}
\left (R^\bullet\pi_*(u^*\mathbb{R}T_\varrho)\right )^\vee,\end{equation}
with $\mathbb{R}T_\varrho$
the relative tangent complex of $\varrho$. 

The following result is part of \cite[Theorem 7.1.6]{CKM}.

\begin{Thm}\label{main}  If $W$ has only lci singularities, then the obstruction theory
(\ref{obs theory}) of $\QmapWe$ is perfect.
\end{Thm}

The lci condition implies that $W$ has a (virtual) tangent bundle $T_W$ in the $\G$-equivariant $K$-group. 
Hence $\det (T_W)$ 
is a well-defined element in $\Pic^{\G}(W)$.
The virtual dimension of $\QmapWe$ is
$$\mathrm{vdim}(\QmapWe)=\beta(\det (T_W)) + (1-g)(\dim(\WmodG)-3)+k.$$
Note that it is independent on $\ke$.

For classes $\beta$ for which the moduli space $\QmapWe$ contains a point represented by a regular map $u:C\lra\WmodG$, we may view $\beta$ as the second
homology class $u_*[C]\in H_2(\WmodG, \ZZ)$ and then we have
$$\beta(\det (T_W))=\int_\beta c_1(T_{\WmodG}).$$
Note, however, that for some $L_\theta$-effective classes it is possible that every point of $\QmapWe$ corresponds to a  quasimap with base-points. The simplest example is
when $W=\CC$, $\G=\CC^*$ with the natural action, linearized by the identity character (so that $\WmodG$ is a single point) and we take $\beta\neq 0$. For a less trivial example,
consider the conic in $\PP^2$ (see Example 2.8.3 below) and take the classes with $\beta(L_\theta)$ a positive odd integer.

\subsection{Graph spaces} \label{action}
We will also need moduli stacks of maps for which the domain curve
contains a {\it parametrized} $\PP^1$, see \S 7.2 of \cite{CKM}. These are called graph spaces. Even though we'll
only use later the genus zero case of the results of this subsection, everything makes sense for arbitrary
genus and we treat things in this generality.

Let $\mathfrak{G}_{g,k,\beta}([W/\G])$ denote the Artin stack  parametrizing tuples
$$((C, x_1,\dots,x_k), P, u, \varphi ),$$
where the new data
$\varphi$ is a morphism from $C$ to $\PP ^1$ of degree one (i.e, such that $\varphi_*[C]=[\PP^1]$).
Alternatively, these are all maps $$C\lra [W/\G]\times \PP^1$$ of class $(\beta, 1)$.
The curve $C$ has one irreducible component $C_0$ such that $\varphi |_{C_0}:C_0\tilde{\ra}\PP^1$
is an isomorphism, while the rest of the curve is contracted by $\varphi$.

By imposing stability conditions as in \S 2 we obtain {\it $\ke$-stable quasimap graph spaces}, denoted
$$QG^{\ke}_{g,k,\beta}(\WmodG).$$ In this case
the ampleness part of the stability condition {\it does not} involve the parametrized component. Precisely, we require as before that
the data $((C, x_1,\dots,x_k), P, u)$ is a prestable quasimap, the length inequality $\ke  \ell(x) \le 1$ for every point $x$ in $C$,
and the modified ampleness condition

$$\omega _{C'} (\sum x_i +\sum y_j) \otimes \mathcal{L}_{\theta} ^\ke {\text{ is ample,}}$$
where $C'$ is the closure of ${C\setminus C_0}$, $x_i$ are the markings on $C'$ and $y_j$ are the nodes $C'\cap C_0$.

In the chamber $1<\ke\leq\infty$ 
we have the usual stable map theory graph space
$$G_{g,k,\beta}(\WmodG)=QG^{\infty}_{g,k,\beta}(\WmodG) =\overline{M}_{g,k}(\WmodG\times\PP^1,(\beta, 1)).$$
For the other extreme chamber corresponding to $\ke= 0+$, the quasimap graph space 
$$QG^{0+}_{g,k,\beta}(\WmodG)$$
was denoted $\mathrm{Qmap}_{g,k}(\WmodG,\beta;\PP^1)$ in \cite[\S7.2]{CKM}.

If we choose coordinates $[\zeta_0,\zeta_1]$ on 
$\PP^1$, then we have the standard $\CC^*$ action, given by
$$t[\zeta_0,\zeta_1]=[t\zeta_0,\zeta_1],\;\;\; t\in\CC^* .$$
It induces actions on the $\ke$-stable quasimap graph spaces which will be
of great importance for the rest of the paper.

The analogues of Theorems \ref{proper} and \ref{main} hold for graph spaces, see \cite[Theorem 7.2.2]{CKM}. 
The only difference is that the moduli stack $\fMgk$ of prestable curves is replaced
by the analogous moduli of prestable curves with one parametrized $\PP^1$ (i.e., by the stack $\fMgk(\PP^1,1)$) in the description of the relative obstruction theory.
With this change, the obstruction theory is still given by the formula \eqref{obs theory}.

\subsection{Quasimap invariants} The condition that base-points do not appear at the markings ensures there are well-defined evaluation maps
$$ev_i :\QmapWe\lra\WmodG,\;\;\; i=1,\dots, k.$$ The moduli spaces also carry the tautological cotangent line bundles at the markings, whose first Chern classes we
denote by $\psi_i$, $i=1,\dots ,k$, as is customary in Gromov-Witten theory. These structures exist on the graph spaces as well.

Assume that $\WmodG$ is projective. For nonnegative integers $a_i$ and cohomology classes $\gamma_i\in H^*(\WmodG,\QQ)$, the
associated descendant $\ke$-quasimap invariant of $\WmodG$ is the rational number
$$\int_{[\QmapWe]^{\mathrm{vir}}}\prod_j  ev_j^*(\gamma_j)\psi_j^{a_j}.$$

If $\WmodG$ is only quasiprojective one may still define a good theory of {\it equivariant} quasimap invariants in the presence of an action with 
good properties by an algebraic torus, 
as explained in \S6.3 of \cite{CKM} and at the beginning of \S\ref{equivariant case} later in this paper. 
 
\subsection{Examples}\label{Examples} We list here several classes of targets for which the theory of quasimaps applies.
\begin{EG} {\bf Smooth toric varieties.} These are GIT quotients $\WmodG$ with $W\cong\CC^N$ a vector space and $\G=(\CC^*)^l$ abelian. 
Their quasimap theory (in the projective case) was
developed first in \cite{CK}. There are also non-compact toric targets with very interesting theories. For example, taking $W=\CC^n\oplus\CC$, $\G=\CC^*$ acting with
weights $(1,\dots,1,-m)$, $m\in\ZZ$, and $\theta$ the identity character gives the GIT quotient $\WmodG=|\cO_{\PP^{n-1}}(-m)|$, the total space of the line bundle $\cO(-m)$ on $\PP^{n-1}$.
\end{EG}

\begin{EG} {\bf Type $A$ flag manifolds.} These are also quotients of a vector space, but with 
$\G$ a product of  $GL_{k_i}(\CC)$'s and $\theta$ given by the determinant character on each factor
(see e.g. \cite{BCK} for this quotient description). In particular, taking $W=\Hom (\CC^r,\CC^n)$ and $\G=GL_r(\CC)$ with the obvious action gives the Grassmannian $G(r,n)$
as $W/\!\!/_{\det}\G$. The quasimap theory of the Grassmannian recovers the {\it stable quotients} theory of \cite{MOP}. Projective space $\PP^{n-1}$ may be viewed either as a special case of flag manifold, or as a toric example.
\end{EG}

\begin{EG} {\bf Projective varieties in general.}
Let $X\subset\PP^{n-1}=\CC^n/\!\!/\CC^*$ be a smooth projective variety. Then $X=W/\!\!/\CC^*$ with $W=C(X)$, the affine cone over $X$ and we have 
the proper moduli spaces of $\ke$-stable quasimaps to $X$. However, for $\ke\leq 1$ the obstruction theory will be perfect precisely when $X$ is a complete intersection.
(For $\ke >1$ and arbitrary $X$ one recovers the usual stable maps to $X$, with their canonical perfect obstruction theory.)
\end{EG}

\begin{EG}\label{zero loci} {\bf Zero loci of sections of homogeneous bundles.}
The previous example of complete intersections in projective space generalizes as follows. Let $W,\G, \theta$ be as in the general set-up and assume that $W$ is smooth.
Let $E$ be a finite dimensional $\G$-representation. Let $W\times E$ be the corresponding $\G$-equivariant vector bundle on $W$ and let 
$${\underline {E}}:=W^s\times_{\G} E
$$
be the induced vector bundle on $\WmodG$. Let $s\in H^0(W,W\times E)^{\G}$ be a $\G$-invariant global section, inducing ${\underline{s}}\in H^0(\WmodG,{\underline{E}})$.
Assume that ${\underline{s}}$ is regular and that its zero locus $X:=Z({\underline{s}})$ is smooth. Then $X=Z/\!\!/_\theta \G$ with $Z:=Z(s)$. Furthermore, $Z$ has only lci
singularities (contained in the unstable locus) hence the moduli spaces of quasimaps to $X$ will carry perfect obstruction theories for all $\ke$.

For toric $\WmodG$ this construction gives precisely the complete intersections in toric varieties (since the representation $E$ splits into one-dimensional summands), but there
are many interesting targets $Z/\!\!/\G$ with $\underline{E}$ non-split when $\G$ is non-abelian. For example, as explained in \cite{BCK},
all flag manifolds of classical types $B,C$, and $D$ are realized 
in this fashion, with $\WmodG$ a flag manifold of type $A$. 
For another example (suggested to us by E. Scheidegger), 
consider on the Grassmannian $G(2,6)$ the vector bundle $\underline{E}=Sym^3(S^\vee)$, where $S^\vee$ is the dual of the tautological rank $2$
sub-bundle. $\underline{E}$ has rank $4$ and first Chern class $6$, and descends from the third symmetric power of the standard representation of $GL_2(\CC)$. The zero locus
of a regular section defines a Calabi-Yau fourfold of the form $Z/\!\!/GL_2(\CC)$ which is not a complete intersection in $G(2,6)$.
\end{EG}

\begin{EG} {\bf Local targets.}\label{local targets} Fix again $W,\G, \theta$ as in the general set-up, and assume $\WmodG$ is projective.
Let $E$ be a $\G$-representation as in the previous example and let $\cE:=W\times E$ denote the associated $\G$-equivariant vector bundle with projection
$$p:\cE\lra W.$$ Assume that $H^0(W,\cE)^\G=0$. (We call such a representation $E$ {\it concave}.
One may take a general $\G$-equivariant bundle $\cE$ on $W$, but most cases of interest arise from $\G$-representations and we restrict to those here for simplicity.)
The line bundle $p^*L_\theta$ gives a linearization of
the $\G$-action on $\cE$ and the quotient $\cE/\!\!/\G$ is the total space of the induced bundle $\underline{E}$ on $\WmodG$. 

In Gromov-Witten theory these kind of
spaces are known as {\it local targets}. Since they are not compact, a good theory of invariants requires the presence of a torus action. Specifically,
consider the torus $\bS=\CC^*$ which acts on $\cE$ by scalar multiplication in
the fibers. Let $\CC_\lambda$ be the standard representation of $\bS$.
Let $0:W\lra\cE$ be the closed embedding as the zero section.
It is $\G$-equivariant and $0^*p^*L_\theta=L_\theta$, so it induces a closed embedding of stacks
$$i:\QmapWe\hookrightarrow \mathrm{Q}^\ke_{g,k}(\cE/\!\!/\G,\beta).$$
(Since the pull-back via $p$ gives an isomorphism
$\Pic^{\G}(W)\cong\Pic^\G(\cE)$, the curve classes are the same for $\WmodG$ and $\cE/\!\!/\G$.)
The following facts are easy to check:

$(i)$ The $\bS$-fixed substack in $\mathrm{Q}^\ke_{g,k}(\cE/\!\!/\G,\beta)$ is
$\QmapWe$. In particular, it is proper.

$(ii)$ The fixed part of the obstruction theory coincides with the canonical obstruction theory
of $\QmapWe$

$(iii)$ The virtual normal bundle is $R^\bullet\pi_*\cE_{g,k,\beta}(\lambda)$, where
$$\cE_{g,k,\beta}(\lambda)=u^*(\fP\otimes_\G(\cE\otimes\CC_\lambda)),$$
with $\pi :\cC\lra\QmapWe$ the universal curve, $\fP$ the universal principal $\G$-bundle on it, and $u$ the universal section.

The $\ke$-quasimap invariants of $\cE/\!\!/\G$ are defined via the virtual localization formula and take values in the field $\QQ(\lambda)$.

\end{EG}

\begin{EG} {\bf Nakajima quiver varieties.} These form a class of non-compact GIT targets with interesting torus-equivariant quasimap theory, see Example 6.3.2 in \cite{CKM}. 
They are holomorphic symplectic varieties, a useful feature when applying to them the general results on genus zero quasimap theory we prove later in this paper, 
see Remark \ref{noncompact} $(ii)$.
\end{EG}

\section{Maps between moduli spaces and universal line bundles}\label{maps} There are several morphisms between various moduli spaces that
we will need. 

\subsection{Embeddings} The polarization $\cO(\theta)$ gives an embedding
$$\iota:\WmodG\lra \PP^N_{A^\G}
=\Proj(\mathrm{Sym}_{A^\G}((\Gamma (W, L_\theta)^\G)^{\vee}))$$
over the affine quotient $\Spec(A^\G)$.
In particular, we have a basis of $\G$-equivariant global sections $t_0,\dots,t_N\in \Gamma (W, L_\theta)^\G$
which cut out the unstable locus $W^u$ set-theoretically:
$$W^u=\{w\in W\; |\; t_i(w)=0, i=0,\dots,N\}.$$

Now $\PP^N_{A^\G}$ with the usual polarization $\cO(1)$ is itself the GIT quotient 
$$(\Spec(A^\G)\times\CC^{N+1})/\!\!/\bH,$$ with $\bH=\CC^*$ acting only on the factor
$\CC^{N+1}$ by
$$\lambda\cdot(z_0,\dots,z_N)=(\lambda z_0,\dots,\lambda z_N),\;\;\; \lambda\in\bH ,$$
and linearization $\CC_{\eta}$.
Here $\CC_{\eta}$ denotes the standard
$1$-dimensional representation of $\bH=\CC^*$, corresponding to the character $\eta= id_{\CC^*}$.
(We use the notation $\bH$ to distinguish this group from the group $\CC^*$ acting on graph spaces
considered in \S\ref{action}.)

Hence we have the graph space
$$\mathfrak{G}_{g,k,d}([\CC_{A^\G} ^{N+1}/\bH ]),$$
whose $\CC$-points we identify as usual with the data
\begin{equation}\label{map to P} ((C,x_1,\dots, x_k), \cL, (s_i)_{i=0,\dots, N},\varphi),\end{equation}
where $\cL$ is a line bundle on $C$ and $s_i\in \Gamma(C,\cL)$ are global sections. The degree $d$ of the map 
$$[(s_0,\dots,s_N)]:C\lra [\CC_{A^\G} ^{N+1}/\bH ]$$ is
in this case the unique homomorphism $$d:\Pic_{\bH}(\CC^{N+1})\cong \ZZ\lra \ZZ$$ with $d(\CC_{\eta})=\deg(\cL)$.
The map (\ref{map to P}) is a quasimap if and only if the restrictions of the sections
$s_j$ to every irreducible component of $C$ are not all identically zero. This description agrees with the earlier one
using principal $\bH$-bundles $P$ on $C$, via the identification $\cL=P\times_\bH \CC_\eta$.

Every $\CC$-point
$$((C, x_1,\dots,x_k), P, u, \varphi )$$
in $\mathfrak{G}_{g,k,\beta}([W/\G])$ induces the tuple
$$((C, x_1,\dots,x_k), u^*(P\times _\G L_{\theta}), u^*(t_i)_{i=0,...,N},\varphi)$$
which is point in $\mathfrak{G}_{g,k,d(\beta)}([\CC_{A^\G} ^{N+1}/\bH ])$.
Here the degree $d(\beta)$ is the homomorphism $$\Pic_{\bH}(\CC^{N+1})\lra\ZZ,\;\;\; \CC_{\eta}\mapsto \beta(u^*(P\times _\G L_\theta)).$$

Since the construction is natural, it will work in families as well, so we obtain a morphism of stacks
$$\iota_{\mathfrak{G}} :\mathfrak{G}_{g,k,\beta}([W/\G])\lra \mathfrak{G}_{g,k,d(\beta)}([\CC_{A^\G} ^{N+1}/\bH ]),$$
which is clearly equivariant for the action of $\CC^*$ on the graph spaces from \S\ref{action}.
It depends obviously on the choice of $\theta$ and the sections $t_0,\dots,t_N$.

It is immediate to check that $\iota_{\mathfrak{G}}$ preserves the $\ke$-stability condition for each $0+\leq\ke\leq\infty$. Hence
there are $\CC^*$-equivariant morphisms
\begin{equation}\label{epsilonqembedding}\iota_{\ke} : QG^\ke_{g,k,\beta}(\WmodG)\lra QG^\ke_{g,k,d(\beta)}(\PP^N_{A^\G}).
\end{equation}
The morphism $\iota_\infty$ is of course the usual embedding of moduli spaces
 obtained by composing a stable map with the given embedding 
 $$\iota\times id_{\PP^1}:\WmodG\times\PP^1\lra \PP^N_{A^\G}\times\PP^1.$$
 
 Finally,  note that the construction works equally for the moduli spaces without a parametrized component,
since the parametrization $\varphi: C\lra\PP^1$ was simply carried along unmodified. We obtain morphisms

\begin{equation}\label{epsilonqembedding2}\iota_\ke:\QmapWe\lra\mathrm{Q}^\ke_{g,k}(\PP^N_{A^\G},d(\beta)). \end{equation}
 
 \subsection{Contractions} We'll also need various contraction morphisms between moduli spaces with target
$$\PP^N_{A^\G}=(\Spec(A^\G)\times\CC^{N+1})/\!\!/\bH.$$

\subsubsection{Contractions to the parametrized $\PP^1$}
The first kind we consider applies to graph spaces and contracts everything to base-points on the parametrized $\PP^1$. Hence the target of all such morphisms will be 
Drinfeld's original ``space of quasimaps" 
$$QG_{0,0,d}(\PP^N_{A^\G})$$  
compactifying the Hom scheme
$\Map _d(\PP^1,\PP^N_{A^\G})$. 
If $\zeta_0,\zeta_1$ are homogeneous coordinates on $\PP^1$ as in \S\ref{fixed loci} and
$$V_{d}\cong H^0(\PP^1,\cO_{\PP^1}(d))$$ denotes the vector space of homogeneous polynomials of 
degree $d$ in $\zeta_0,\zeta_1$,
then $\Map _d(\PP^1,\PP^N_{A^\G})$ can be identified with 
$$(\mathcal{U}/\bH)\times\Spec(A^\G),$$
where 
$$\mathcal{U}\subset V_d^{\oplus(N+1)}$$ 
is the open subset of $(N+1)$-tuples of polynomials with no common zero on $\PP^1$. The Drinfeld
compactification is simply the projectivization
$$\PP(V_d^{\oplus(N+1)})\times\Spec(A^\G).$$

It is very easy to describe these morphisms point-wise. Let  
$$((C,x_1,\dots, x_k), \cL, (s_i)_{i=0,\dots, N},\varphi:C\lra\PP^1)$$
be a $k$-pointed $\ke$-stable quasimap to $\PP^N_{A^\G}$, of genus $g\geq 0$ and degree $d$, with a parametrized $\PP^1$
in the domain.

Let $C_0$ be the parametrized component of $C$ 
and let 
$y_1,\dots,y_m\in C_0$ be the points where $C_0$ meets the rest of the curve. Since $\varphi$ contracts
all other irreducible components of $C$,
$$\overline{C\setminus C_0}=C_1\coprod\dots\coprod C_m$$
is the union of $m$ connected components.
For $i=0,1,\dots, m$, let $d_i$ be the degree of the induced quasimap with domain curve $C_i$.
Since the class of any quasimap is $L_\theta$-effective, the divisor 
$$D:=\sum_{i=1}^m d_iy_i $$
on $C_0$ is effective. Consider the degree $d$ line bundle 
$$\tilde{\cL}:=(\cL |_{C_0})(D)$$ on $C_0$.
Composing the sections $s_j|_{C_0}$ with the  injection
$$0\lra\cL |_{C_0}\lra(\cL |_{C_0})(D)$$
gives $N+1$ sections $\tilde{s}_0,\dots,\tilde{s}_N$ of $\tilde{\cL}$. 
Note that all sections $\tilde{s}_j$ vanish at each $y_i$ and that the cokernel of
$$\cO_{C_0}^{\oplus(N+1)}\stackrel{(\tilde{s}_0,\dots,\tilde{s}_N)}{\lra}\tilde{\cL}$$
has length $d_i$ at $y_i$.
The data
$$(C_0,\tilde{\cL},\tilde{s}_0,\dots,\tilde{s}_N)$$
represents a point in $QG_{0,0,d}(\PP^N_{A^\G})$.

Informally, each connected component $C_i$, $i=1,\dots, N$ is replaced by a
base-point at its attaching point on $C_0$, which keeps track only of the degree of $C_i$,
then all remaining markings (if any) on $C_0$ are forgotten.

A proof that the above construction can be performed in families in a functorial
way can be 
found in \cite{LLY}. Strictly speaking, the proof is given there only for families of stable maps to $\PP^N$,
but one checks easily that the argument applies equally well for $\ke$-stable quasimaps. We conclude
that for each $0+\leq\ke \leq \infty$ there is a morphism of stacks
\begin{equation}\label{epsilonqcontraction1}
\Phi_\ke: QG^\ke_{g,k,d}(\PP^N_{A^\G})\lra QG_{0,0,d}(\PP^N_{A^\G}).
\end{equation}
These morphisms are clearly equivariant for the $\CC^*$-action on graph spaces.

\subsubsection{Contractions of rational tails}\label{rat-tails}
The second kind of contractions are morphisms 
\begin{equation}\label{contraction2}
c: \overline{M}_{g,k}(\PP^N_{A^\G},d)\lra \mathrm{Q}_{g,k}(\PP^N_{A^\G},d)
\end{equation}

\begin{equation}\label{graphcontraction2}
c: G_{g,k,d}(\PP^N_{A^\G})\lra QG_{g,k,d}(\PP^N_{A^\G})
\end{equation}
from stable maps moduli to stable
quasimaps moduli which replace the ``rational tails" of a stable map with base points, keeping track of their degrees.
They appeared already in \cite{MM} and are described point-wise also in \cite{MOP} for the case
without a parametrized component in the domain curve. 
For example, we spell out the case with a parametrized $\PP^1$.  

Let  
$$((C,x_1,\dots, x_k), \cL, (s_i)_{i=0,\dots, N},\varphi:C\lra\PP^1)$$
be a stable map with a parametrized distinguished component.
Let $T_1,..., T_l$ be the maximal connected trees of rational curves in the domain curve C satisfying the properties
\begin{enumerate}
\item $T_i$ contains no markings and the distinguished component $C_0$ is not a component in $T_i$;
\item $T_i$ meets the rest of the curve $C$ in a single point $z_i$.
\end{enumerate}
Let $d_i$ be the degree of the restriction to $T_i$ of the stable map.
Denote by $\tilde{C}$ the closure of $C\setminus(\cup_{i=1}^l T_i)$; by assumption, it contains the distinguished $C_0$ as a component. 
We then have the line bundle
$$\tilde{\cL}:=\cL|_{\tilde{C}}\otimes\cO_{\tilde{C}}(\sum_{i=1}^l d_iz_i)$$ on $\tilde{C}$,
with sections $\tilde{s}_0,\dots,\tilde{s}_N$ obtained as the compositions
$$\tilde{s}_j:\cO_{\tilde{C}}\stackrel{s_j}{\lra}\cL|_{\tilde{C}}\lra \cL|_{\tilde{C}}\otimes\cO_{\tilde{C}}(\sum_{i=1}^l d_iz_i).$$
The image of the stable map under the contraction morphism \eqref{graphcontraction2} is the stable quasimap
$$((\tilde{C},x_1,\dots, x_k), \tilde{\cL},(\tilde{s}_i)_{i=0,\dots, N},\varphi|_{\tilde{C}}).
$$
Again, it is standard to see that the constructions can be done functorially in families.

We have described the morphisms between the moduli spaces corresponding to the two extremal chambers for the stability parameter $\ke$.
In fact, these can be easily seen to factor into a composition of contractions
$$c_{\ke_1,\ke_2}:\mathrm{Q}^{\ke_1}_{g,k}(\PP^N_{A^\G},d)\lra \mathrm{Q}^{\ke_2}_{g,k}(\PP^N_{A^\G},d)$$
(and similar ones for graph spaces) for $\ke_1>\ke_2$ in adjacent chambers, see \cite{MM}, \cite{Toda}.
We will not really use these morphisms in the present paper, but they are important for comparing higher genus $\ke$-stable quasimap invariants
as $\ke$ varies, see \S\ref{higher genus} and \cite{MOP}, \cite{Toda}, \cite{CKtoric2}.

\subsubsection{Replacing markings by base-points} \label{marking to bp}
Another variant of the above constructions will be useful to us. Let  
$$((C,x_1,\dots, x_k, y_1,\dots,y_m), \cL, (s_i)_{i=0,\dots, N})$$
be a $k+m$-pointed $\ke$-stable quasimap to $\PP^N_{A^\G}$, of genus $g\geq 0$ and degree $d$. Further, fix positive integers $d_1,\dots d_m$. The data
\begin{equation}\label{prest}((C,x_1,\dots, x_k), \cL\otimes\cO_C(\sum_{j=1}^m d_jy_j), (\tilde{s}_i)_{i=0,\dots, N})
\end{equation}
is a $k$-pointed prestable quasimap to $\PP^N_{A^\G}$, where the sections $\tilde{s}_i$ are again obtained by composing $s_i$ with the injection
$$\cL\lra\cL\otimes\cO_C(\sum_{j=1}^m d_jy_j).
$$
 For each $j$ the quasimap has a base-point of length $d_j$ at $y_j$. It is possible that some of the rational components of $C$ are now rational tails, i.e., they contain
 a single node and none of the markings $x_i$. Each such rational tail is part of a unique maximal tree as in the previous subsection. Let $T_1,..., T_l$ be the set of all such trees
 and let $e_1,\dots e_l$ be their total degrees (that is, $e_j$ is the degree of the restriction to the subcurve $T_j$ of the quasimap \eqref{prest}). We now contract these trees to
 base-point of length $e_j$, just as in \S\ref{rat-tails} to obtain a $k$-pointed stable (in the sense of $\ke=0+$-stability) quasimap to $\PP^N_{A^\G}$. The upshot is that we obtain
 for each $0+\leq\ke\leq \infty$ a morphism
 \begin{equation} b_{\ke,\{d_j\}}:\mathrm{Q}^{\ke}_{g,k+m}(\PP^N_{A^\G},d)\lra\mathrm{Q}_{g,k}(\PP^N_{A^\G},d+\sum d_j)
 \end{equation}
 and similarly for graph spaces,
 \begin{equation}b_{\ke,\{d_j\}}:QG^\ke_{g,k+m,d}(\PP^N_{A^\G})\lra QG_{g,k,d+\sum d_j}(\PP^N_{A^\G}).
 \end{equation}

\subsection{Universal line bundles}\label{universal} For each $0+\leq\ke\leq \infty$ consider the composed $\CC^*$-equivariant morphisms
$$\Phi_\ke\circ\iota_\ke : QG^{\ke}_{g,k,\beta}(\WmodG) \lra QG_{0,0,d(\beta)}(\PP^N_{A^\G}).$$

Recall that we have the identification
$$QG_{0,0,d(\beta)}(\PP^N_{A^\G})=\PP(V_{d(\beta)}^{\oplus(N+1)})\times\Spec(A^\G),$$
where $V_{d(\beta)}\cong H^0(\PP^1,\cO_{\PP^1}(d(\beta)))$ is the vector space of homogeneous polynomials of 
degree $d(\beta)$ in two variables $\zeta_0,\zeta_1$. Furthermore, under this identification, the action of $\CC^*$ 
on $V_{d(\beta)}$ is given by
multiplication on the first variable $\zeta_0$ only. 
It has the canonical lifting to an action on the
(relative) tautological line bundle $\cO(1)$ on $\PP(V_{d(\beta)}^{\oplus(N+1)})\times\Spec(A^\G)$. 
In this paper, we always consider this canonical $\CC^*$-linearization of $\cO(1)$.

Following Givental, we define the {\it universal $\CC^*$-equivariant line bundle} $U(L_\theta)$ on $QG^{\ke}_{g,k,\beta}(\WmodG)$ by
\begin{equation}\label{universal1}U(L_\theta):=(\Phi_\ke\circ\iota_\ke)^*\cO(1).\end{equation}

Similarly, given effective classes $\beta_1,\dots \beta_m$ we have a universal line bundle on $QG^\ke_{g,k+m,\beta}(\WmodG) $, still denoted by $U(L_\theta)$ 
and defined as the
pull-back of $\cO(1)$ via 
$$\Phi_{\ke}\circ b_{\ke, \{d(\beta_j)\}}\circ\iota_{\ke} : QG^\ke_{g,k+m,\beta}(\WmodG) \lra QG_{0,0,d(\beta)+\sum d(\beta_j)}(\PP^N_{A^\G}).$$

As ordinary line bundles, these $U(L_\theta)$ on graph spaces can in fact be defined without the use of the contraction morphisms, 
via the universal principal $\G$-bundle and the representation $\CC_\theta$; however it will be important in our arguments that they
have the {\it canonical} $\CC^*$-linearization induced from the linearization of $\cO(1)$ described above.

\section{$\CC^*$-localization on graph spaces}
One of the most important feature of the graph spaces is that
the natural $\CC^*$-action on them has fixed point loci described via moduli spaces of {\it unparametrized}
quasimaps, with easy to compute virtual normal bundles. Via the virtual localization theorem of \cite{GP}, this description leads to 
useful factorization properties of generating functions for the evaluations of certain cohomology classes on graph spaces against
the virtual class. We recall
in this section the description of the fixed point loci and of the equivariant Euler classes of their virtual normal bundles. 

\subsection{Fixed loci of $\CC^*$-action on graph spaces}\label{fixed loci} 

Fix $g,k\geq 0$, and an $L_\theta$-effective  class $\beta$. For $\ke \in [0+,\infty]$ we will consider moduli spaces
$$\Qglonge  $$
and think of $\bullet$ as a distinguished (last) marking. Note that for $\beta=0$ and $2g+k\geq 2$
we have
$$\Qglonge\cong \overline{M}_{g,k+\bullet}\times\WmodG $$
and all $k+1$ evaluation maps at markings are the projection onto the second factor.

If an $\ke$-stable parametrized quasimap
$$((C, x_1,\dots,x_k), P, u, \varphi )$$
is $\CC^*$-fixed, then all markings, all nodes, all base-points, and the entire
degree $\beta$ must be supported over the fixed points $0$ and $\infty$ on the parametrized component $C_0$.
We have
$$(\QGge)^{\CC^*}=\coprod {F}_{g_2,k_2,\beta_2}^{g_1,k_1,\beta_1},$$
the union over all possible splittings
$$g=g_1+g_2,\;\; k=k_1+k_2,\;\; \beta=\beta_1+\beta_2,$$
with $g_i,k_i\geq 0$ and $\beta_i$ effective. 

The component ${F}_{g_2,k_2,\beta_2}^{g_1,k_1,\beta_1}$ 
is identified with the fiber product
\begin{equation}\label{simple loci}
 \Qgonelonge\times_{\WmodG}\Qgtwolonge 
 \end{equation}
 over the evaluation maps $\eb$ at the special markings $\bullet$.
Under this identification, the inclusion
$$i: F_{g_2,k_2,\beta_2}^{g_1,k_1,\beta_1}\lra \QGge$$ is obtained as
follows: given two $\ke$-stable quasimaps to $\WmodG$ sending the special markings $\bullet$ to the same point $p$, we glue them at $0$ and $\infty$ to the
constant map $\PP^1\lra p\in\WmodG$.

The unstable cases $(g_2,k_2,\beta_2)=(0,1,0)$ or $(g_2,k_2,\beta_2)=(0,0,0)$ (and likewise for $(g_1,k_1,\beta_1)$) are included
above by the following conventions:
$$\mathrm{Q}^\ke_{0,0+\bullet}(\WmodG,0):=\WmodG,\;\;\mathrm{Q}^\ke_{0,1+\bullet}(\WmodG,0):=\WmodG,\;\;\; \eb :=id_{\WmodG}.$$
This is because 
$$QG_{0,1,0}(\WmodG)\cong\WmodG\times\PP^1$$
with fixed point loci $\WmodG\times\{0\}$ and $\WmodG\times\{\infty\}$, while
$$QG_{0,0,0}(\WmodG)\cong\WmodG$$ with trivial $\CC^*$-action.

For stable maps ($\ke=\infty$) the above description covers all possibilities. However, for quasimaps there are more unstable
cases which need to be treated separately, namely $(g_2,k_2,\beta_2)=(0,0,\beta_2)$ with $\beta_2\neq 0$ and $\ke\leq\frac{1}{\beta_2(L_\theta)}$
(and the symmetric ones $(0,0,\beta_1\neq 0)$). This is because in genus zero moduli of stable quasimaps in the chamber $\ke=0+$ require 
at least two markings.
For $\beta\neq 0$, we denote by
$$\mathrm{Q}_{0,0+\bullet}(\WmodG,\beta)_\infty$$
the moduli space parametrizing the quasimaps of class $\beta$
$$(\PP^1,P,u)
$$
with $P$ a principal $\G$-bundle on $\PP^1$, $u:\PP^1\lra P\times_\G W$ a section such that $u(x)\in W^s$ for $x\neq\infty\in \PP^1$ and 
$\infty\in \PP^1$ is a base-point of length $\beta(L_\theta)$. Similarly, we have the moduli space
$$\mathrm{Q}_{0,0+\bullet}(\WmodG,\beta)_0$$
for which $u$ has now a base-point of length $\beta(L_\theta)$ at $0\in \PP^1$ and no other base points. (Note that the two moduli spaces are isomorphic, but
they sit in $QG_{0,0,\beta}(\WmodG)$ as different components of the $\CC^*$-fixed locus, with different normal bundles.)

In either case the regular map 
$$u_{reg}:\PP^1\lra\WmodG$$
induced by the quasimap is a constant map. We define the evaluation map at the special marking by
$$ev_\bullet ((\PP^1,P,u))=u_{reg}(\PP^1).$$
With these definitions we have 
$${F}_{0,0,\beta_2}^{g,k,\beta_1}\cong\mathrm{Q}^\ke_{g,k+\bullet}(\WmodG,\beta_1)\times_{\WmodG}\mathrm{Q}_{0,0+\bullet}(\WmodG,\beta_2)_\infty
$$
for $k\geq 1$ and $\ke\leq\frac{1}{\beta_2(L_\theta)}$ and similarly
$${F}_{g,k,\beta_2}^{0,0,\beta_1}\cong\mathrm{Q}_{0,0+\bullet}(\WmodG,\beta_1)_0\times_{\WmodG}\mathrm{Q}^\ke_{g,k+\bullet}(\WmodG,\beta_2)
$$
for $k\geq 1$ and  $\ke\leq\frac{1}{\beta_1(L_\theta)}$. When $g=k=0$ and $\ke\leq\mathrm{min}\{\frac{1}{\beta_1(L_\theta)},\frac{1}{\beta_2(L_\theta)}\}$,
$${F}_{0,0,\beta_2}^{0,0,\beta_1}\cong\mathrm{Q}_{0,0+\bullet}(\WmodG,\beta_1)_0\times_{\WmodG}\mathrm{Q}_{0,0+\bullet}(\WmodG,\beta_2)_\infty .
$$

\subsection{Euler classes of virtual normal bundles} For the rest of the section we restrict to genus zero. 
We drop the genus from the notation and write $F_{k_2,\beta_2}^{k_1,\beta_1}$ for the components of the fixed point locus in $\QGraphe$. 
The Euler classes of their virtual normal bundles are essentially the product of contributions from $(k_1,\beta_1)$ and $(k_2,\beta_2)$. 
The notation $H^*(X)$ will mean cohomology with $\QQ$-coefficients throughout the rest of the paper.

We begin with the loci $F_0:=F_{0,0}^{k,\beta}$ where everything is concentrated over $0\in\PP^1$. In this case only $(k,\beta)$ contributes to the product.
As explained in the previous subsection, there are two cases to consider. 

Assume first that either $k\geq 1$, or if $k=0$,
that $\ke>\frac{1}{\beta(L_\theta)}$ (note that
this is {\it always} true in for stable maps). Hence $F_0\cong\Qlonge$.
If $N_{F_0}$ is the {\it virtual} normal bundle, given by the moving part of the (absolute) obstruction theory, then its $\CC^*$-equivariant Euler class
is
\begin{equation}\label{cont k>0}
\mathrm{e}_{\CC^*}(N_{F_0})=\mathrm{cont}_{(k,\beta)}(z):=\begin{cases}
1, & (k,\beta)=(0,0),\\
z, & (k,\beta)=(1,0),\\
z(z-\psi_\bullet) , & \text{otherwise},
\end{cases}\end{equation}
where $z$ denotes the generator of $A^*_{\CC^*}(\Spec(\CC))$. 
The formula \eqref{cont k>0} is well-known, see e.g., \cite{Givental} for stable maps
and \cite {CK}, \cite{CKM} for quasimaps. Note that it is the same for all targets $\WmodG$.

In addition, we observe that in these cases the $\CC^*$-fixed part of the obstruction theory for the graph space gives precisely the obstruction theory of $\Qlonge$.

On the other hand, if $k=0$ and $\ke\leq\frac{1}{\beta(L_\theta)}$, then $$F_0\cong\mathrm{Q}_{0,0+\bullet}(\WmodG,\beta)_0.$$
In this case
\begin{equation}\label{cont k0}\mathrm{e}_{\CC^*}(N_{F_0})=\mathrm{cont}_{(0,\beta)}(z)\in H^*(F_0,\QQ)\otimes \QQ(z)\end{equation}
will depend on the target and needs to be calculated separately. In fact, as we will see in the next section, what one is interested in calculating is the push-forward
to $H^*(\WmodG,\QQ)\otimes\QQ(z)$ of $[F_0]^{\mathrm{vir}}/ \mathrm{cont}_{(0,\beta)}(z)$ via the evaluation map $\eb$ (these will essentially be the coefficients of the
so-called small $I$-function of $\WmodG$). For many of the examples of targets listed in \S\ref{Examples} closed formulas for these push-forwards are known, see
Remark \ref{small I functions} below.

For a general component $F_{k_2,\beta_2}^{k_1,\beta_1}$, we have
\begin{equation}\label{contproduct}\mathrm{e}_{\CC^*}(N_F)=\mathrm{cont}_{(k_1,\beta_1)}(z)\boxtimes\mathrm{cont}_{(k_2,\beta_2)}(-z),\end{equation}
with each factor as in \eqref{cont k>0} or \eqref{cont k0}.

\section{Generating functions for genus zero quasimap invariants}\label{J and S} In this section we require that quasimap invariants are well-defined.
Hence we will assume either that $\WmodG$ is projective,
or that $\WmodG$ has 
an additional action by a torus ${\T}$, with proper fixed loci in the quasimap moduli spaces, as in \S6.3 of \cite{CKM}.
\subsection{The $J^\ke$ and $I$-functions}  For varying $k$ and $\beta$, we assemble into generating functions
the residues of the virtual classes of 
$\QGraphe$
at the fixed loci 
$F_0={F}_{0,0}^{k,\beta}$, 
paired against cohomology classes pulled back from $\WmodG$ via
the evaluation maps at the $k$ markings, then push-forward to $H_*(\WmodG)$ using the evaluation $\eb$ at the special markings.

Introduce the Novikov ring 
$$\Lambda=\Lambda((W,\G,\theta)):=\{ \sum_{\beta\in \mathrm{Eff}(W,\G,\theta)} a_\beta q^\beta | a_\beta\in \QQ\},$$
the $q$-adic completion of the semigroup ring $\QQ[\mathrm {Eff}(W,\G,\theta)]$.  In the quasi-projective case we extend the coefficients
to the field $$K:=\QQ(\{\lambda_j\})=H^*_{{\T},\mathrm{loc}}(\Spec(\CC),\QQ),$$
the {\it localized} $\T$-equivariant cohomology of a point.
Sometimes we will also use the notation $\QQ[[q]]$ (or $K[[q]]$) for the Novikov ring.

The generating functions will be formal functions of
${\bf t}\in H^*(\WmodG)$. (In the case of quasi-projective targets we take ${\bf t}$ in the ${\T}$-equivariant cohomology of $\WmodG$.
To simplify notation, in this section we will only write $H^*(\WmodG)$ in either instance; the torus equivariant case is discussed
in more detail in section \ref{equivariant case}.)

\begin{Def} Fix $\ke \in \QQ_{>0}$. The (big) $J^\ke$-function of $\WmodG$ is
\begin{align}\label{Je}
J^\ke({\bf t}, z)&=&\sum_{k\geq 0,\beta\geq 0}q^\beta(\eb)_*\left(\frac{\prod_{i=1}^k ev_i^*({\bf t})}{k!}\cap\mathrm{Res}_{F_0}[\QGraphe]^{\mathrm{vir}}\right)\\
\nonumber&=& \one+ \frac{{\bf t}}{z} + \sum _{k=0, \beta \neq 0, \beta(L_\theta)\leq 1/\ke}
q^\beta (\eb)_*\frac{[\mathrm{Q}_{0,0+\bullet}(\WmodG,\beta)_0]^{\mathrm{vir}}}{\mathrm{e}_{\CC^*}(N_{F_0})}\\
\nonumber&+& \sum_{\stackrel{k\geq 1\; {\mathrm{or}}\; \beta(L_\theta)> 1/\ke ,} {(k,\beta)\neq (1,0)}}q^\beta (\eb)_*\frac{\prod_{i=1}^k ev_i^*({\bf t})\cap[\Qlonge]^{\mathrm{vir}}}{z(z-\psi_\bullet)}.
\end{align}
\end{Def}
Here $\one=[\WmodG]\in H^0(\WmodG)$ is the unit in cohomology. The properness of the quasimap moduli spaces over the affine quotient is easily seen to imply that the evaluation maps $\eb$ are proper, hence the push-forwards are 
well-defined for all targets, projective or not.

For any $\infty\ge\ke >1$ the above definition gives the usual big $J$-function in Gromov-Witten theory:
\begin{align}\label{J>1}J({\bf t}, z)&=J^{\infty}({\bf t}, z)=\one+ \frac{{\bf t}}{z} \\
\nonumber &+ \sum _{(k,\beta)\neq(0,0),(1,0)}q^\beta (\eb)_*\left(\frac{\prod_{i=1}^k ev_i^*({\bf t})}{k!}\cap\frac{[\Mlong]^{\mathrm{vir}}}{z(z-\psi_\bullet)}\right).
\end{align}

On the other hand we may also consider the $(\ke=0+)$-stability condition to obtain a similar generating function, called {\it the big $I$-function of} $\WmodG$ and introduced
in \cite{CK},\cite{CKM}:

\begin{align}\label{J0+} I({\bf t}, z)&=J^{0+}({\bf t}, z)=\one+ \frac{{\bf t}}{z} + \sum_{k=0, \beta \neq 0}
q^\beta (\eb)_*\frac{[\mathrm{Q}_{0,0+\bullet}(\WmodG,\beta)_0]^{\mathrm{vir}}}{\mathrm{e}_{\CC^*}(N_{F_0})}\\
\nonumber  &+\sum_{k\geq 1,(k,\beta)\neq(1,0)}q^\beta (\eb)_*\left(\frac{\prod_{i=1}^k ev_i^*({\bf t})}{k!}\cap\frac{[\Qlong]^{\mathrm{vir}}}{z(z-\psi_\bullet)}\right).
 \end{align}
 
For each $\ke\geq 0+$, the function $J^\ke$ takes values in 
$$H^*(\WmodG,\Lambda)\{\!\{z,z^{-1}\}\!\},$$ 
where we have used Poincar\' e duality to pass to cohomology. (For quasi-projective targets the coefficient ring is replaced with the localized $\T$-equivariant cohomology 
$H^*_{{\T},\mathrm{loc}}(\WmodG,\Lambda)$. ) By definition, $H^*(\WmodG,\Lambda)\{\!\{z,z^{-1}\}\!\}$ is a certain completion of the ring of formal Laurent series in $1/z$. Precisely,
\begin{equation}\label{completion}
\{ \sum_{j=-\infty}^{\infty} h_jz^j \; |\; h_j\lra 0\; {\text {as}}\; j\ra\infty\; {\text{ in the}}\; q{\text {-adic topology on}}\; \Lambda \}.
\end{equation}
In fact this completion is needed only for the $I$-function $I=J^{0+}$, as positive powers of $z$ can possibly appear only in the terms
with $k=0, \beta \neq 0$, and $\beta(L_\theta)\leq 1/\ke$.

The first two asymptotic properties below 
are easily obtained from the definitions; 
the third follows from a dimension counting argument (in fact, we will see a more general statement in \S\ref{semi-positive} below).
\begin{enumerate}
\item $J^\ke=e^{{\bf t}/ z}+O(q)$ for all $\ke\geq 0+$.
\item $J^{\infty}=\one+\frac{{\bf t}}{z}+O(\frac{1}{z^2})$.
\item Assume that $\WmodG$ is  Fano of index at least $2$, i.e., that for every $\beta\neq 0$ we have $\beta (\det T_W)\geq 2$. Then
$J^{\ke}=\one+\frac{{\bf t}}{z}+O(\frac{1}{z^2})$ for all $\ke\geq 0+$.
\end{enumerate}

In the rest of this section we derive other properties of these generating functions from the geometry of the $\CC^*$-action on graph spaces.

\subsection{Polynomiality}\label{poly} 
The following construction in Gromov-Witten theory is due to Givental; we just observe here that it applies to quasimap theory as well, 
for all $\ke\geq0+$.

Consider the universal line bundle $U(L_\theta)$ on $\QGraphe$ defined in \eqref{universal1}.
Using the identification of each $F_{k_2,\beta_2}^{k_1,\beta_1}$ with a fibered product we immediately see that
\begin{equation}\label{Urestriction} U(L_\theta)|_{F_{k_2,\beta_2}^{k_1,\beta_1}}
= ev ^*_{\bullet} (\cO(\theta) ) \boxtimes \CC _{\beta_2 (\theta)}.\end{equation}
Define a generating function for primary quasimap integrals on graph spaces by
\begin{equation}\label{graph potential}
D^\ke:=\sum _{k,\beta\geq 0}q^\beta\int_{[\QGraphe]^{\mathrm{vir}}}\frac{e^{c_1^{\CC^*}(U(L_\theta))y}\prod_{i=1}^k ev_i^*({\bf t})}{k!}.
\end{equation}
Here $y$ is a formal variable, $c_1^{\CC^*}$ is the equivariant first Chern class, and the symbol $\int$ denotes $\CC^*$-equivariant push-forward to a point. 

Fix once and for all a homogeneous basis
$\{\gamma_i\}_i$ of $H^*(\WmodG)$ and denote by $\{\gamma^i\}_i$ the dual basis with respect to the intersection pairing on $\WmodG$,
$$\int_{\WmodG}\gamma_i\gamma^j=\delta_i^j.
$$
(As usual, in the quasi-projective $\T$-equivariant case the integral is understood via the localization theorem as an integral over the $\T$-fixed locus in $\WmodG$, which is
proper.)

If we write ${\bf t}=\sum t_i\gamma_i$, then $D^\ke$ is an element in the ring of formal power series 
$\Lambda[[t_i, y, z]]$. Moreover, for each $\beta$ the coefficient of $q^\beta$ in the right-hand side is a polynomial in $z$.

We now apply the virtual localization theorem \cite{GP} to $D^\ke$. Using the product expression \eqref{contproduct} for the virtual normal bundles and the restriction
formula \eqref{Urestriction} we obtain the following factorization.

\begin{Prop}\label{poly lemma}Set
$(qe^{-zyL_\theta})^\beta:=q^\beta e^{-zy\beta(L_\theta)}$. Then
$$D^\ke=\int _{\WmodG} e^{c_1(\cO(\theta)) y} J^\ke(q, {\bf t}, z) J^\ke(qe^{-zyL_\theta}, {\bf t}, -z).$$
 In particular, for each $\beta$ the coefficient of $q^\beta$ in the right-hand side is a polynomial in $z$.
\end{Prop}

\subsection {The $S$-operator}

We extend the customary Gromov-Witten theory bracket notation to $\ke$-quasimap invariants: 
for nonnegative integers $a_j$ and cohomology classes $\sigma_j\in H^*(\WmodG)$ 
put
$$\langle \sigma_1\psi_1^{a_1},\dots,\sigma_k\psi_k^{a_k}
\rangle^\ke_{g,k,\beta}:=\int_{[\QmapWe]^{\mathrm{vir}}}\prod_j  ev_j^*(\sigma_j)\psi_j^{a_j}.$$
Extending by linearity, the cohomological insertions $\sigma_j$ may be taken from $H^*(\WmodG,\Lambda)\{\!\{z,z^{-1}\}\!\}$.

Similarly, we can consider invariants on the graph spaces

$$\langle \sigma_1,\dots,\sigma_k
\rangle^{QG^\ke}_{k,\beta}:=\int_{[\QGraphe]^{\mathrm{vir}}}\prod_j  \tilde{ev}_j^*(\sigma_j).$$

This time the evaluation maps $\tilde{ev}_j$ take values in $\WmodG\times \PP^1$ and the cohomology classes 
$$\sigma=\gamma\ot p\in H^*(\WmodG)\ot H^*_{\CC^*}(\PP^1)\cong H^*_{\CC^*}(\WmodG\times\PP^1)$$
are pulled-back from there. 
Again, by linearity, we may use insertions with $\gamma\in H^*(\WmodG,\Lambda)$.
As in the previous subsection, the virtual class is the $\CC^*$-equivariant one and the integral is the equivariant push-forward to a point.
We will write $\gamma p$ for $\gamma\ot p$ and simply $\gamma$ for $\gamma\ot 1$.

Further, we use double brackets to denote generating series of invariants:
$$\lla \sigma_1\psi_1^{a_1},\dots,\sigma_k\psi_k^{a_k} \rra^\ke_{0,k}:=\sum_{m\geq 0,\beta\geq 0} \frac{q^\beta}{m!}
\langle \sigma_1\psi_1^{a_1},\dots,\sigma_k\psi_k^{a_k},
{\bf t},\dots ,{\bf t}\rangle^\ke_{0,k+m,\beta}
$$

$$\lla \sigma_1,\dots,\sigma_k\rra^{QG^\ke}_k:=\sum_{m\geq 0,\beta\geq 0} \frac{q^\beta }{m!}\langle \sigma_1,\dots,\sigma_k, {\bf t},\dots ,{\bf t}
\rangle^{QG^\ke}_{k+m,\beta},
$$
where as before ${\bf t}=\sum_i t_i\gamma_i\in H^*(\WmodG)$. 

 Note that all terms in the right-hand side of the second formula are well-defined, even when $k=0$, while 
in the first formula we must restrict the summation to the {\it stable} cases (or introduce by hand terms that correspond to triples $(k+m,\beta,\ke)$ which are unstable).
For example, we can write the $J$-functions (using the projection formula) as
\begin{equation}\label{bracketJ}
J^\ke=\sum_i\gamma_i\lla \frac{\gamma^i}{z(z-\psi)}\rra^\ke_{0,\bullet}
\end{equation}
with the understanding that the unstable terms are defined as in \eqref{Je}, by residues on graph spaces:
$$\begin{cases}1& m=0,\beta=0\\
{{\bf t}}/{z}& m=1, \beta=0\\
\int_{[F_0]^{\mathrm{vir}}}\frac {\eb^*\gamma^i}{{\mathrm{e}_{\CC^*}(N_{F_0})}}& m=0,\beta\neq 0, \beta(L_\theta)\leq 1/\ke
\end{cases}
$$

The following operator on $H^*(\WmodG,\Lambda)\{\!\{z,z^{-1}\}\!\}$, denoted $S^\ke(z)$ and defined by
\begin{align}
S^\ke(z)(\gamma)&=\sum_{m\ge 0,\beta\ge 0}\frac{q^\beta}{m!} 
(ev_1)_*\left(\frac{[\mathrm{Q}^\ke_{0,2+m}(\WmodG,\beta)]^{\mathrm{vir}}}{z-\psi}ev_2^*(\gamma)\prod_{j=3}^{2+m}ev_j^*({\bf t})\right)\\
\nonumber &=\sum_i\gamma_i \lla \frac{\gamma^i}{z-\psi},\gamma\rra^\ke_{0,2}.
\end{align}
will be very important for the rest of the paper. 
Here the unstable term in $\lla ... \rra^\ke_{0,2}$ corresponding to $m=0,\beta=0$ is defined to be $\int_{\WmodG}\gamma^i\gamma$.
In fact, we have a family of linear operators depending on the parameter ${\bf t}\in H^*(\WmodG)$; we use the notation $S^\ke_{\bf t}(z)$
if the ${\bf t}$-dependence needs to be emphasized.

By the well-known calculation of $\psi$ integrals on $\overline{M}_{0,m+2}$ we have the asymptotic expansion
$$S^\ke(z)(\gamma)=e^{\frac{\bf t}{z}}\gamma+O(q).
$$
When $\infty\geq\ke >1$, this operator is a familiar object in Gromov-Witten theory: its components
$$
S^{\infty}_{ij}(z)=\lla \frac{\gamma^i}{z-\psi},\gamma_j\rra^{\infty}_{0,2}
$$
form a fundamental solution to the quantum differential equation, see \cite{Givental-equiv}. It is well-known that $S^{\infty}(z)$
is a unitary operator.
We prove here that this property holds for all $\ke\geq 0+$.

\begin{Prop}\label{Unitary}
Consider the operator
\begin{equation*}
(S^\ke)^\star(-z)(\gamma)=\sum_i\gamma^i\lla \gamma_i,\frac{\gamma}{-z-\psi}\rra^\ke_{0,2}
\end{equation*}
Then
\begin{equation}\label{unitary}
(S^\ke)^\star(-z)\circ S^\ke(z)=Id.
\end{equation}

\end{Prop}
\begin{proof}
Consider two equivariant cohomology classes $p_0,p_\infty\in H^*_{\CC^*}(\PP^1)$, defined by specifying their restrictions at the fixed points:
\begin{equation}\label{equi insertions}
p_0|_0=z,\; p_0|_\infty=0,\;\;\; {\text {and}}\;\;\; p_\infty|_0=0,\; p_\infty|_\infty=-z.
\end{equation}

By definition, for arbitrary cohomology classes $\gamma,\delta\in H^*(\WmodG)$, the generating series 
$$\lla \gamma p_0,\delta p_\infty\rra^{QG^\ke}_2=\sum_{m\geq 0,\beta\geq 0} \frac{q^\beta }{m!}
\langle \gamma p_0, {\bf t},\dots ,{\bf t},\delta p_\infty \rangle^{QG^\ke}_{2+m,\beta}
$$
is an element in $\Lambda[[z]]$. 
We now calculate it with the virtual localization theorem. Due to the insertions of $p_0$ and $p_\infty$, only the fixed loci $F^{k_1,\beta_1}_{k_2,\beta_2}$
for which the first marking lies over $0\in\PP^1$
and the last marking lies over $\infty\in\PP^1$ have nonvanishing contribution to the integral $\langle \gamma p_0, {\bf t},\dots ,{\bf t},\delta p_\infty \rangle^{QG^\ke}_{2+m,\beta}$.
In view of the description \eqref{simple loci} of the fixed point loci, $\CC^*$-localization gives
$$\lla \gamma p_0,\delta p_\infty\rra^{QG^\ke}_2=
\sum_i \lla \frac{\gamma^i}{z-\psi},\gamma\rra^\ke_{0,2} \lla\delta, \frac{\gamma_i}{-z-\psi}\rra^\ke_{0,2}.
$$
Now the right-hand side of the above equality has the form
$$\int_{\WmodG}\gamma\delta +O(1/z).
$$
while the left-hand side does not have a pole at $z=0$. We conclude that both sides are {\it constant} in $z$ and equal to $\int_{\WmodG}\gamma\delta $.

The Proposition follows now immediately:
\begin{align*}\left[ (S^\ke)^\star(-z)\circ S^\ke(z)\right](\gamma)
&=\sum_j\gamma^j\sum_i \lla \frac{\gamma^i}{z-\psi},\gamma\rra^\ke_{0,2} \lla\gamma_j, \frac{\gamma_i}{-z-\psi}\rra^\ke_{0,2}\\
&=\sum_j\gamma^j(\int_{\WmodG}\gamma\gamma_j)=\gamma .
\end{align*}
\end{proof}

\begin{Rmk} The usual proof of the unitary property in Gromov-Witten theory uses the string equation, as well as the splitting property of Gromov-Witten invariants, 
see \cite{Givental-elliptic}. While quasimap invariants satisfy splitting, the usual form of the string equation fails in general, the reason being that the universal curve  
over $\QmapWe$ is {\it not} isomorphic to $\mathrm{Q}^\ke_{g,k+1}(\WmodG,\beta)$. Nevertheless, the universal curve $\cC_{g,k,\beta}$ 
{\it is} a proper Deligne-Mumford substack in
$\fM _{g,k+1}(\WmodG,\beta)$ (corresponding to a different stability condition), and one may write a modified version of the string equation involving intersection numbers on 
$\cC$. This modified equation suffices to adapt the Gromov-Witten argument to obtain Proposition \ref{unitary}. We  prefer the proof given here as it is another illustration
of the usefulness of $\CC^*$-localization on graph spaces in the genus zero theories. 
\end{Rmk}

\subsection{The $P$-series and Birkhoff factorization of $J^\ke$} \label{Birkhoff}
Next we consider yet another generating function on graph spaces:

\begin{equation}\label{P-series def}
P^\ke({\bf t}, z):=\sum_i\gamma^i \lla \gamma_ip_\infty\rra^{QG^\ke}_1=\sum_i\gamma^i\sum_{m\geq 0,\beta\geq 0} \frac{q^\beta }{m!}
\langle \gamma_ip_\infty, {\bf t},\dots ,{\bf t}\rangle^{QG^\ke}_{1+m,\beta}.
\end{equation}
Being $\CC^*$-equivariant integrals, the graph space brackets take values in $\QQ[z]$, hence $P^\ke({\bf t}, z)$ is an element of $H^*(\WmodG,\Lambda)[[z]]$ for each ${\bf t}$ (convergent in the $q$-adic topology in the sense of \eqref{completion}).

Again, we apply virtual localization to $P^\ke$. The calculation is similar to the one done in the proof of Proposition \ref{Unitary}. In this case, the fixed loci that contribute
are those of the form $F^{k_1,\beta_1}_{1+k_2,\beta_2}$, for which the first 
marking lies over $\infty\in\PP^1$.

\begin{align*}P^\ke&=\sum_i\gamma^i\sum_j \lla \frac{\gamma^j}{z(z-\psi)}\rra^\ke_{0,1}\lla(-z)\gamma_i, \frac{\gamma_j}{-z(-z-\psi)}\rra^\ke_{0,2}\\
&=\sum_i\gamma^i\lla \gamma_i, \frac{\sum_j\gamma_j\lla \frac{\gamma^j}{z(z-\psi)}\rra^\ke_{0,1}}{-z-\psi}\rra^\ke_{0,2}\\
&=\sum_i\gamma^i\lla \gamma_i, \frac{J^\ke}{-z-\psi}\rra^\ke_{0,2}\\
&=(S^\ke)^\star (-z)(J^\ke),
\end{align*}
where we used the equality \eqref{bracketJ}, with the convention explained there for the unstable terms in $ \lla \frac{\gamma^j}{z(z-\psi)}\rra^\ke_{0,1}$.
Note that the relation
\begin{equation}\label{PfromJ}
P^\ke=(S^\ke)^\star (-z)(J^\ke).
\end{equation}
shows in particular that $P^\ke$ has the asymptotic form
$$P^\ke=\one+O(q).
$$
(This can easily be seen directly by calculating the $\beta=0$ integrals.
Indeed, for $\beta=0$ and $m\geq 0$, there is an isomorphism $${QG^\ke}_{1+m,0}(\WmodG)\cong
\WmodG\times\PP^1[m+1],$$
where $\PP^1[m+1]\cong\overline{M}_{0,m+1}(\PP^1,1)$ is the Fulton-MacPherson space of stable marked genus zero curves with a rigid component. Hence the 
 $\beta=0$ integrals are of the form
$$\left( \int_{\WmodG}\gamma_i{\bf t}^m\right )\left(\int_{\overline{M}_{0,m+1}(\PP^1,1)}ev_1^*(p_\infty)\right).$$
The second factor can be written as $\int_{\PP^1}(ev_1)_*(\one)\cap p_\infty$, hence it vanishes except when $m=0$, in which case $ev_1$ is an isomorphism and
the integral is equal to $1$. We deduce that the only nonvanishing contribution is $\sum_i\gamma^i\left(\int_{\WmodG}\gamma_i\right)=\one$.)

Using the unitary property in equation \eqref{PfromJ} we obtain the following Theorem.

\begin{Thm} \label{Birkhoff Thm}
For every $\ke\geq 0+$
$$J^\ke(z)=S^\ke(z)(P^\ke)=\sum_i S^\ke(z)(\gamma^i)\lan \gamma_i,P^\ke\ran,
$$
where $\lan -,-\ran$ is the intersection pairing $\int_{\WmodG}$.
\end{Thm}

Observe that while the $J$-functions will in general depend both on positive and negative powers of $z$ when $\ke\leq 1$, 
the operator $S^\ke$ is a series of $1/z$ and $P^\ke$ is a series of $z$. Hence we may view Theorem \ref{Birkhoff Thm} as a kind of Birkhoff factorization 
of the $J^\ke$-functions.

On the other hand, for $\ke >1$, the usual Gromov-Witten $J$-function is only a series of $1/z$, with asymptotic expansion $\one+{\bf t}/z+ O(1/z^2)$. 
This implies that $P^{\ke >1}=\one$ and the Theorem becomes the
familiar result
$$J=S(z)(\one),$$
which follows immediately from the string equation.

\subsection{The case of semi-positive targets} \label{semi-positive} 
We assume in this subsection that the triple $(W,\G,\theta)$ satisfies $$\beta(\det(T_W))\geq0$$
for all $L_\theta$-effective classes $\beta$ and will call such triples {\it semi-positive}. 

In the projective case, the semi-positive triples give GIT quotients with nef anti-canonical class. 
Among quasiprojective targets, the main examples we have 
in mind are holomorphic symplectic quotients, such as Nakajima quiver varieties, and local targets $\cE/\!\!/\G$ as in Example \ref{local targets} with 
$0\leq \beta (\det T_W)+\beta(\det E)$ for all $L_\theta$-effective $\beta$.

First we specialize to the case ${\bf t}=0$. Since the virtual dimension of $QG^\ke_{0,1,\beta}(\WmodG)$
is equal to
$$\dim (\WmodG)+1+\beta (\det T_W)$$
and the insertion $\gamma_i p_\infty$ has complex degree at most $\dim (\WmodG)+1$, we conclude that 
$$\lan \gamma_i, P^\ke|_{{\bf t}=0}\ran=0$$
unless $\gamma_i=[\mathrm{pt}]$ is the point class. 
Furthermore, $\lan [\mathrm{pt}],  P^\ke|_{{\bf t}=0}\ran$ is an element of $\Lambda=\QQ[[q]]$, of the form 
\begin{equation}\label{J0}
\lan [\mathrm{pt}],  P^\ke|_{{\bf t}=0}\ran=1+\sum_{\{\beta\neq 0, \beta(\det T_W)=0\}}a_\beta q^\beta.
\end{equation}
In particular, it is invertible, and it is equal to $1$ if $\beta(\det T_W)\geq 1$ for all effective $\beta\neq 0$. We will say that the triple $(W,\G,\theta)$ is a {\it Fano triple} in this case. 

Let $S^\ke_0(z)$ denote the operator $S^\ke$ at ${\bf t}=0$. Then
\begin{align*}S^\ke _0(z)(\one)&=\one+\sum_i\gamma_i(\sum_{\beta\neq 0}q^\beta\lan\frac{\gamma^i}{z-\psi},\one\ran_{0,2,\beta}^\ke)\\&=
\one+\frac{1}{z}\sum_i\gamma_i\left(\sum_{\beta\neq 0}q^\beta \lan \gamma^i,\one\ran_{0,2,\beta}^\ke\right)+O(\frac{1}{z^2}).
\end{align*}

\begin{Cor}\label{nef small} For semi-positive triples $(W,\G,\theta)$ and every $\ke \geq 0+$
\begin{equation}
\frac{J^\ke |_{{\bf t}=0}}{\lan [\mathrm{pt}],  P^\ke|_{{\bf t}=0}\ran}=\one+\sum_i\gamma_i(\sum_{\beta\neq 0}q^\beta\lan\frac{\gamma^i}{z-\psi},\one\ran_{0,2,\beta}^\ke).
\end{equation}
\end{Cor}
\begin{proof} Set ${\bf t}=0$ in Theorem \ref{Birkhoff Thm}. \end{proof}

\begin{Rmk} In the very special case when $\WmodG$ is a nef complete intersection in projective space $\PP^n$ and $\ke=0+$, 
Corollary \ref{nef small} was obtained independently in \cite{CZ}.
Their proof is different, as it uses localization with respect to the big torus action on $\PP^n$.
\end{Rmk}

By their definition \eqref{Je}, the big $J^\ke$-functions
have the form
$$\frac{{\bf t}}{z}+J^\ke |_{{\bf t}=0}+O(1/z^2).$$
If we write the $1/z$-expansion of $J^\ke |_{{\bf t}=0}$ as
$$J^\ke |_{{\bf t}=0}=J^\ke_0(q)\one+J^\ke_1(q)\frac{1}{z}+O(1/z^2),$$
then we deduce from Corollary \ref{nef small} that 
\begin{equation}J^\ke_0(q)=\lan [\mathrm{pt}],  P^\ke|_{{\bf t}=0}\ran=1+O(q)\in \Lambda,\end{equation}
and
\begin{equation}\label{J1}
J^\ke_1(q)=J^\ke_0(q)\sum_i\gamma_i\left(\sum_{\beta\neq 0}q^\beta \lan \gamma^i,\one\ran_{0,2,\beta}^\ke\right)\in qH^*(\WmodG,\Lambda).
\end{equation}
Hence
\begin{equation}\label{J-asy}
J^\ke({\bf t},z)=J^\ke_0(q)\one+\frac{{\bf t}+J^\ke_1(q)}{z}+O(1/z^2).
\end{equation}
This implies in particular that 
$$P^\ke({\bf t},z)=(S^\ke(-z))^\star(J^\ke({\bf t},z))=J^\ke_0(q)\one+O(1/z).$$
Since $P^\ke$ has no pole at $z=0$, it follows that $P^\ke({\bf t},z)=J^\ke_0(q)\one$. From this and Theorem \ref{Birkhoff Thm} we deduce the following 
extension of Corollary \ref{nef small}:
\begin{Cor}\label{nef big}
For semi-positive $(W,\G,\theta)$ we have
\begin{equation} \frac{J^\ke({\bf t},z)}{J^\ke_0(q)}=S_{\bf t}^\ke(z)(\one)=\sum_i\gamma_i\lla\frac{\gamma^i}{z-\psi},\one\rra_{0,2}^\ke.
\end{equation}
Furthermore, 
\begin{equation}\label{mirror 1}\frac{{\bf t}+J^\ke_1(q)}{J_0^\ke(q)}={\sum_i\gamma_i\lla\gamma^i,\one\rra_{0,2}^\ke}-\one.\end{equation}
\end{Cor} 

Equation \eqref{mirror 1} allows us to easily determine, in the semi-positive case, all primary $\ke$-quasimap invariants with a fundamental class insertion in terms of the first
two coefficients $J^\ke_0(q)$ and $J^\ke_1(q)$ of the $1/z$ expansion of $J^\ke |_{{\bf t}=0}$. 

\begin{Cor}\label{whatever} Assume that $(W,\G,\theta)$ is semi-positive. Then

$(a)$ For every $i$ and every $m\geq 2$,
$$\sum_{\beta\neq 0}q^\beta \lan \gamma^i,\one, {\bf t},\dots,{\bf t}\ran_{0,2+m,\beta}^\ke=0.$$

$(b)$ For every $i$,
$$\sum_{\beta\neq 0}q^\beta \lan \gamma^i,\one\ran_{0,2,\beta}^\ke=\lan \frac{J^\ke_1(q)}{J^\ke_0(q)},\gamma^i\ran.$$

$(c)$ If $j\neq i$, then
$$\sum_{\beta\neq 0}q^\beta \lan \gamma^i,\one, \gamma_j\ran_{0,3,\beta}^\ke=0,$$
while
$$1+\sum_{\beta\neq 0}q^\beta \lan \gamma^j,\one, \gamma_j\ran_{0,3,\beta}^\ke=(J^\ke_0(q))^{-1}$$
for every $j$.

\end{Cor}
\begin{proof} Write ${\bf t}=\sum_it_i\gamma_i$ as before. Part $(a)$ follows directly from \eqref{mirror 1}, since the left-hand side is linear in ${\bf t}$. Part $(b)$ is \eqref{J1}.
Using $(a)$ and $(b)$ in \eqref{mirror 1}, we get
$${\bf t}=J^\ke_0(q)\left ({\bf t}+\sum_i\gamma_i\sum_{\beta\neq 0}q^\beta \lan \gamma^i,\one, {\bf t}\ran_{0,3,\beta}^\ke\right).$$
Fixing $j$ and making $t_i=0$ for $i\neq j$ gives part $(c)$.
\end{proof}

In fact, simple degree counting gives a more precise version of part $(b)$.
For any target $\WmodG$ we give a grading to $H^*(\WmodG,\Lambda)\{\!\{z,z^{-1}\}\!\}$ as follows:
\begin{itemize} 
\item elements in $H^*(\WmodG)$ have their usual cohomological degree;
\item $q^\beta$ has degree $2\beta(\det(T_W))$;
\item $z$ has degree 2.
\end{itemize}
It is immediate to check that the operator $S^\ke_0(z)$ preserves degree, hence $S^\ke_0(z)(\one)$ is an element of degree zero. 

Assume now that $\WmodG$ is projective, with $(W,\G,\theta)$ semi-positive.
As we already observed, $J_0^\ke(q)$ has degree zero. 
The $1/z$-coefficient in $S^\ke_0(z)(\one)$ is
$$\sum_i\gamma_i\left(\sum_{\beta\neq 0}q^\beta \lan \gamma^i,\one\ran_{0,2,\beta}^\ke\right)=\frac{J^\ke_1(q)}{J^\ke_0(q)},$$
which therefore has degree $2$. Hence the only possibly non-vanishing invariants $\lan \gamma^i,\one\ran_{0,2,\beta}^\ke$ are those for which 
$\gamma^i\in H^{\ge \dim_{\RR}(\WmodG)-2}(\WmodG)$.
Write
$$J^\ke_1(q)=f_0^\ke(q)\one+\sum_{j=1}^r f_j^\ke(q)D_j$$
where $f_0^\ke(q)\in q\Lambda$ is a degree $2$ element, 
$$f_0^\ke(q)=\sum_{\{\beta\neq 0, \beta(\det(T_W))=1\}}b_\beta q^\beta,$$
$\{D_1,\dots ,D_r\}$ is a basis of $H^2(\WmodG)$, and $f_j^\ke(q)\in q\Lambda$ are of degree zero, 
$$f_j^\ke(q)=\sum_{\{\beta\neq 0, \beta(\det(T_W))=0\}}c_{j,\beta} q^\beta.$$
We conclude
$$\sum_{\beta\neq 0}q^\beta \lan [\mathrm{pt}],\one\ran_{0,2,\beta}^\ke=\frac{f_0^\ke(q)}{J^\ke_0(q)},\;\;\; 
\sum_{\beta\neq 0}q^\beta \lan D^j,\one\ran_{0,2,\beta}^\ke=\frac{f_j^\ke(q)}{J^\ke_0(q)},$$
where $\{D^j\}$ is the dual basis in $H^{\dim_{\RR}(\WmodG)-2}(\WmodG)$.

Note that 
if $(W,\G,\theta)$ is Fano, then $J_0^\ke(q)=1$ and $J_1^\ke(q)=f_0^\ke(q)\one$, 
and if the Fano index is at least $2$ (i.e., if $\beta(\det T_W)\geq 2$ for all effective $\beta\neq 0$), 
then $J_1^\ke(q)=0$.

\begin{Rmk}\label{small I functions} The main point of the formulas in Corollary \ref{whatever} is that most often 
$J^\ke |_{{\bf t}=0}$ is explicitly
computable. 
Indeed, for many of the examples of targets listed in \S\ref{Examples} closed formulas for the ``small" $I$-function $I |_{{\bf t}=0}=J^{0+} |_{{\bf t}=0}$ are known:
\begin{itemize}
\item for toric varieties and complete intersections in toric varieties the formulas are due to Givental, see \cite{Givental};
\item for flag varieties of classical types $A,B,C$ and $D$, as well as zero
loci of regular sections of homogeneous vector bundles on such flag varieties, see \cite{BCK} and \cite{CKS};
\item the $I$-function of a local target $\cE/\!\!/\G$ over a base $\WmodG$ can be written explicitly in terms of the $I$-function of the base 
(this is easy when the bundle splits into a direct sum of line bundles and can be seen by ``abelianization" in the general case).
\item a formula in the case of the Hilbert schemes of points in $\CC^2$ (an example of Nakajima quiver variety) is due to the authors together with Diaconescu and Maulik, see \cite{CKP} for an account; the method generalizes to other Nakajima quiver varieties.
\end{itemize}
Of course, the formulas in the case $\ke=0+$ include the ones for any fixed $\ke>0$, 
since the corresponding $J^\ke_0(q)$ and $J^\ke_1(q)$
are just truncations of $I_0(q)$ and $I_1(q)$, respectively.

\end{Rmk}

\begin{Rmk}\label{noncompact} The situation is  even simpler in the non-compact examples.

(i) Consider a local Calabi-Yau target $\cE/\!\!/\G$ as in Example \ref{local targets} with 
$0= \beta (\det T_W)+\beta(\det E)$ for all effective $\beta$. Since the top degree of a cohomology class is strictly less than
${\mathrm{dim}}(\cE/\!\!/\G)$, the dimension estimate used to conclude \eqref{J0} implies that $J_0^\ke(q)=1$.

(ii) Consider a Nakajima quiver variety $\WmodG$ and let $-\lambda$ denote the $\T$-weight of its
holomorphic symplectic form. Then
$$J_0^\ke(q)=1,\;\;\;\; J_1^\ke(q)=\lambda f^\ke(q)$$
with $ f^\ke(q)\in q\QQ[[q]]$ a (non-equivariant) scalar function of degree zero.
This follows from the same dimension counting arguments and the fact that the moduli spaces of $\ke$-stable quasimaps (and the graph spaces as well) admit
{\it reduced} virtual classes $[\QmapWe]^{\mathrm{vir}}_{\mathrm{red}}$ of dimension one larger than the usual virtual dimension and satisfying
$$[\QmapWe]^{\mathrm{vir}}=\lambda [\QmapWe]^{\mathrm{vir}}_{\mathrm{red}}.$$
In terms of {\it reduced} quasimap invariants we have explicitly
$$f^\ke(q)=\sum_{\beta\neq 0} q^\beta \int_{[Q^\ke_{0,2}(\WmodG,\beta)]^{\mathrm{vir}}_{\mathrm{red}}} ev_1^*([{\mathrm {pt}}]) ev_2^*(\one).$$

The reduced virtual classes for stable maps to holomorphic symplectic varieties were introduced in \cite{MP}, \cite{OP}. The construction works equally well for quasimaps.

\end{Rmk}

\section{Genus zero $\ke$-wall-crossing}\label{J}

In this section we formulate genus zero wall-crossing formulas for quasimap invariants when varying the $\ke$-stability parameter as equality of generating functions
after certain transformations, or changes of variables.
\subsection{``String" transformations}
Let $$\gamma\in H^*(\WmodG,\Lambda)$$ 
be an invertible element of the form $\gamma=\one +O(q)$. (In the quasi-projective cases we may also take  $\gamma=\gamma_0 +O(q)$, with
$\gamma_0$ an invertible element in
$H^*_{\T,\mathrm{loc}}(\WmodG,\QQ)$.)
For each $\ke\geq 0+$ consider the generating series $\sum_i\gamma_i \lla \gamma^i,\gamma\rra^\ke_{0,2}$ for the {\it primary} $\ke$-quasimap invariants with a $\gamma$-insertion and put
\begin{equation}\label{tau e}
\tau^\ke_\gamma({\bf t}):=\sum_i\gamma_i \lla \gamma^i,\gamma\rra^\ke_{0,2}-\gamma.
\end{equation}
We have the $z$-expansion 
$$S^\ke_{\bf t}(z)(\gamma)=\gamma +\tau^\ke_\gamma({\bf t})\frac{1}{z}+O(\frac{1}{z^2}).$$
Further, 
$${\bf t}\mapsto \tau^\ke_\gamma({\bf t})={\bf t}\gamma+O(q)={\bf t}+O(q),$$
hence $\tau^\ke_\gamma$ is invertible as a transformation on $H^*(\WmodG,\Lambda)$. If for a pair of stability parameters $0+\leq\ke_1<\ke_2\leq \infty$ we
set 
\begin{equation}\label{gen string}
\tau^{\ke_1,\ke_2}_\gamma ({\bf t}) := (\tau^{\ke_1}_\gamma)^{-1}\circ \tau^{\ke_2}_\gamma ({\bf t}),
\end{equation}
then
\begin{equation}\label{1/z matching}
S^{\ke_1}_{\tau^{\ke_1,\ke_2}_\gamma ({\bf t})}(z)(\gamma)=S^{\ke_2}_{\bf t}(z)(\gamma)\;\;\;\; ({\mathrm{mod}}\; \frac{1}{z^2}).
\end{equation}

\begin{Conj}\label{conj S} For all $\gamma\in H^*(\WmodG,\Lambda)$ of the form $\one +O(q)$, and all $0+\leq\ke_1<\ke_2\leq \infty$
$$S^{\ke_1}_{\tau^{\ke_1,\ke_2}_\gamma ({\bf t})}(z)(\gamma)=S^{\ke_2}_{\bf t}(z)(\gamma).
$$
\end{Conj}

For example, let us take $\ke_1=\ke$ with $0+\leq\ke\leq 1$, $\ke_2=\infty$, and $\gamma=\one$. Then, by the string equation in Gromov-Witten theory, $\tau^\infty_\one$ is the identity transformation and 
$S^{\infty}_{\bf t}(z)(\one)=J^{\infty}({\bf t},z)$.
Hence
\begin{equation}\label{mirror map}
\tau^{\ke,\infty}_\one = (\tau^{\ke}_\one)^{-1}
\end{equation}
and we obtain the following special case of Conjecture \ref{conj S}:
\begin{Conj}\label{mirror J1} Let $J^\infty({\bf t},z)$ be the Gromov-Witten theory (big) $J$-function of $\WmodG$, let $\ke\geq 0+$ be a stability parameter, and let
\begin{equation}\label{string transf}\tau^{\ke} ({\bf t})=\sum_i\gamma_i \lla \gamma^i,\one\rra^\ke_{0,2}-\one={\bf t}+\sum_i\gamma_i\sum_{\beta\neq 0,m\geq 0}
\frac{q^\beta}{m!}\lan\gamma^i,\one,{\bf t},\dots ,{\bf t}\ran^\ke_{0,2+m,\beta} .\end{equation}
Then
$$J^\infty(\tau^{\ke}({\bf t}),z)=S^\ke_{{\bf t}}(z)(\one).
$$
\end{Conj}

Note that if the $\ke$-quasimap invariants would satisfy the usual string equation, then $\tau^{\ke} ({\bf t})={\bf t}$. 
For this reason we will refer to \eqref{string transf} as {\it the string transformation}. By the same token, \eqref{gen string} will be called a
{\it generalized string transformation}.

\subsection{Semi-positive targets and the mirror map} \label{semi-positive2} 
We specialize further to the semi-positive case. Using the results of \S\ref{semi-positive}, namely 
Corollary \ref{nef big}, together with  Corollary \ref{whatever} and the discussion immediately following it, we restate Conjecture \ref{mirror J1} as follows:

\begin{Conj}\label{mirror nef} Assume that $(W,\G,\theta)$ is semi-positive. Then for every $\ke\geq 0+$
\begin{equation}\label{equation mirror nef}J^\infty(\tau^{\ke} ({\bf t}),z)=\frac{J^\ke({\bf t},z)}{J_0^\ke(q)},
\end{equation}
with 
$$\tau^{\ke} ({\bf t})=\frac{{\bf t}+J^\ke_1(q)}{J_0^\ke(q)}.$$ 
In particular, 

$(i)$  $J^\ke$ is on the overruled Lagrangian cone $\cL ag_{\WmodG}$ encoding the 
genus zero Gromov-Witten theory of $\WmodG$. (See for example \cite{CG} for Givental's symplectic
space formalism in Gromov-Witten theory and the resulting Lagrangian cone.)

$(ii)$ If the Fano index of $(W,\G,\theta)$ is at least $2$, then $J^\ke$ does not depend on $\ke\geq 0+$.
\end{Conj}

\begin{Rmk}\label{main remark} We explain here why the above statement generalizes Givental's toric Mirror Theorems in \cite{Givental}. 

Consider the (most interesting) case when $\ke=0+$, so that the right hand side of \ref{equation mirror nef} is given by the $I$-function of $\WmodG$.
We then set ${\bf t}=0$ and use  the string and divisor equations for Gromov-Witten invariants in the left-hand side (recall from \S\ref{semi-positive} that $J^\ke_1(q)\in H^{\leq 2}(\WmodG,\Lambda)$). When $\WmodG$ is either a semi-positive toric variety, or a semi-positive complete intersection in a toric variety, or a local toric target, the resulting
equality coincides precisely with Givental's toric Mirror Theorems 
after multiplying by an overall factor $e^{{\bf t}_{\mathrm{small}}/z}$, where 
${\bf t}_{\mathrm{small}}$ is the general element of $H^{\leq 2}(\WmodG,\QQ)$.

From this we see that the classical mirror map is 
\begin{equation}\label{classical}
{\bf t}_{\mathrm{small}}\mapsto {\bf t}_{\mathrm{small}}+ \frac{I_1(q)}{I_0(q)}={\bf t}_{\mathrm{small}}+\sum_i\gamma_i\sum_{\beta\neq 0}
q^\beta\lan\gamma^i,\one\ran^{0+}_{0,2,\beta}
\end{equation}
and in particular its ``quantum corrections" have an interpretation as two-pointed quasimap invariants (the second equality is given by Corollary \ref{nef small}).

Hence Conjecture \ref{mirror J1} , and the
Theorems proved in \S\ref{equivariant case} below, can be viewed as a significant generalization of Givental's results: 
\begin{itemize} 
\item from complete intersections
in toric varieties or local toric targets to a much larger class of GIT targets;
\item from semi-positive targets to targets with no restriction on the anti-canonical class;
\item from ${\bf t}=0$ to general parameter ${\bf t}$, i.e., from $2$-point invariants to invariants with any number of insertions;
\item from $(\ke=0+)$-stability to all stability parameters $\ke$.
\end{itemize}
Furthermore, the string transformation $\tau^{\ke} ({\bf t})$ generalizes the 
celebrated ``mirror map", while at the same time giving it a geometric realization as a generating series of
quasimap invariants.
\end{Rmk}

\begin{Rmk} In \cite{Jin}, M. Jinjenzi proposed a conjecture for writing the mirror map via two-pointed correlators in the toric cases which is similar to the second equality in \eqref{classical}, though it is not
clear that the moduli spaces he is using coincide with the moduli of quasimaps as constructed in \cite{CK}. 
Hence Corollary \ref{nef small} proves a generalization of his conjecture to all GIT targets. 

\end{Rmk}

\subsection{Higher genus}\label{higher genus}
 We plan to study wall-crossing formulas under change of the stability parameter $\ke$ for higher genus quasimap invariants in subsequent work. 
Here we only formulate a conjecture in the case when the anti-canonical class is sufficiently positive.
\begin{Conj}
If the Fano index of $(W,\G,\theta)$ is least 2, then the descendant quasimap invariants 
$$\langle \sigma_1\psi_1^{a_1},\dots,\sigma_k\psi_k^{a_k}
\rangle^\ke_{g,k,\beta}$$
are independent on $\ke$.
\end{Conj} 

The Conjecture is inspired by the following cases which are known to hold 
\begin{itemize}
\item When $\WmodG$ is a Grassmannian $G(r,n)$ \cite{MOP} proves it for $\ke\in\{0+,\infty\}$ and \cite{Toda} extends it to arbitrary $\ke$. 
\item When $\WmodG$ is a projective Fano toric variety (with no restriction on the Fano index), the Conjecture is proved in \cite{CKtoric2}.
\end{itemize}

\subsection{Transformations for the big $J^\ke$-functions in the general case}\label{big J wall-crossing} In this subsection we generalize Conjecture \ref{mirror nef} 
to targets with no positivity restriction on the triple $(W,\G,\theta)$. 

\begin{Lemma}\label{algebraic transformation} 
For every $\ke\geq 0+$ there exist a uniquely determined element $P^{\infty,\ke}({\bf t},z)\in H^*(\WmodG,\Lambda)[[z]]$, convergent in the
$q$-adic topology for each ${\bf t}$, and a uniquely
determined transformation ${\bf t}\mapsto\tau^{\infty,\ke}({\bf t})$ on
$H^*(\WmodG,\Lambda)$ with the following properties:
\begin{enumerate}
\item $\tau^{\infty,\ke}({\bf t})={\bf t}+ O(q).$
\item $P^{\infty, \ke}({\bf t},z)=\one +O(q).$
\item $S^\ke_{\bf t}(z)(P^\ke({\bf t}, z))=S^\infty_{\tau^{\infty,\ke}({\bf t})}(z)(P^{\infty, \ke}(\tau^{\infty,\ke}({\bf t}),z))\;\; ({\mathrm{mod}}\; \frac{1}{z^2})$.
\end{enumerate}

\end{Lemma}
We omit the proof of the Lemma, which is an elementary but tedious check 
that $P^{\infty, \ke}$ and $\tau^{\infty,\ke}$ with the required properties can be constructed in an unique way and at the same time by an inductive procedure in the ``degree''
$d=\beta(L_\theta)$.

By Theorem \ref{Birkhoff Thm} we have $J^\ke({\bf t},z)=S^\ke_{\bf t}(z)(P^\ke({\bf t}, z))$. On the other hand, if we write 
$$P^{\infty,\ke }({\bf t},z)=\sum_i C_i(q,\{t_j\}, z)\gamma_i$$
with ``scalar" coefficients $C_i(q,\{t_j\}, z)\in\Lambda[[\{t_j\},z]]$, then
$$S^\infty_{\bf t}(z)(P^{\infty, \ke}({\bf t},z))=\sum_i C_i z\partial_{t_i}J^\infty({\bf t},z),$$
and hence it is on the Lagrangian cone $\cL ag_{\WmodG}$.

\begin{Conj}\label{big J} For all $\ke\geq 0+$, 
$$J^\ke({\bf t},z)=S^\infty_{\tau^{\infty,\ke}({\bf t})}(z)(P^{\infty,\ke}(\tau^{\infty,\ke}({\bf t}),z)).$$
In particular, the function $J^\ke$ is on the Lagrangian cone of the Gromov-Witten theory of $\WmodG$.
\end{Conj}

\section{GIT quotients with good torus actions}\label{equivariant case}

In this section we assume that a torus ${\T} \cong (\CC^*)^r$ acts on $W$ and this action commutes with the $\G$-action. Hence there are induced $\T$-actions on
the quotients $[W/\G]$, $W/_{\mathrm {aff}}\G$, and $\WmodG$. Further, we assume that the $\T$-fixed locus $(\WmodG)^{\T}$ is a finite set. The same is then true
for the fixed locus in the affine quotient, which implies in turn that the $\T$-fixed loci in all quasimap moduli spaces are proper. Hence we have well-defined theories of
$\T$-equivariant $\ke$-quasimap invariants, see \S6.3 in \cite{CKM}.
Furthermore, we may (and will) consider {\it twisted} theories, as in \S6.2 of \cite{CKM}.

\subsection{Equivariant generating functions} Let
$$H^*_{\T}(\Spec(\CC),\QQ)=\QQ[\lambda_1,\dots,\lambda_r]$$
be the equivariant cohomology of a point. The equivariant quasimap invariants are defined by the virtual localization formula: 
$$\langle \sigma_1\psi_1^{a_1},\dots,\sigma_k\psi_k^{a_k}
\rangle^\ke_{g,k,\beta}:=\int_{[(\QmapWe)^{\T}]^{\mathrm{vir}}}\frac{\iota^*(\prod_j  ev_j^*(\sigma_j)\psi_j^{a_j})}{\mathrm{e}(N^{\mathrm{vir}})},$$
where $$\iota:(\QmapWe)^{\T}\hookrightarrow\QmapWe$$
is the inclusion of the fixed point locus and $\mathrm{e}(N^{\mathrm{vir}})$ is the $\T$-equivariant Euler class of the virtual normal bundle.
These are well defined by the properness of the fixed locus and take values in the localized equivariant cohomology of a point,
$$K:=\QQ(\lambda_1,\dots,\lambda_r)=H^*_{\T,\mathrm{loc}}(\Spec(\CC),\QQ).$$ 
Similarly, we have the $\T$-equivariant invariants on graph spaces. 

The Novikov ring is now $\Lambda=K[[q]]$
and we have the equivariant objects $J^\ke({\bf t},z)$, $S_{\bf t}^\ke(z)$, $P^\ke({\bf t},z)$ depending on ${\bf t}\in H^*_{\T,\mathrm{loc}}(\WmodG,\Lambda)$ and taking values in
$H^*_{{\T},\mathrm{loc}}(\WmodG,\Lambda)\{\!\{z,z^{-1}\}\!\}$.

If $\WmodG$ is {\it projective} and we take all insertions $\gamma,{\bf t}$, etc. in the non-localized equivariant cohomology $H^*_{\T}(\WmodG,\QQ[[q]])$, then the 
$\T$-equivariant invariants
may be defined without localization
and take values in the ring $\QQ[\lambda_1,\dots,\lambda_r]$. In this case, the generating functions specialize to their non-equivariant counterparts upon setting 
all $\lambda_i=0$.

\subsection{Twisted theories}\label{twisting} 
Twisted quasimap invariants are discussed in \S6.2 of \cite{CKM} and we refer the reader to loc. cit. for the general theory. We recall here the 
specific situations needed for the present paper.
\subsubsection{Zero loci of sections of ``positive" vector bundles and twisting} \label{twisted 1}
Consider the situation described in Example \ref{zero loci}: $W$ is smooth, $\WmodG$ is projective, and we have a vector bundle $\underline{E}$ on 
$\WmodG$ with a regular section $\underline{s}$, induced from a $\G$-representation 
$E$ and a $\G$-invariant section $s$ of $W\times E$. 
We are interested in the genus zero quasimap theory of $Z/\!\!/\G$, where $Z$ is the zero locus of $s$. Note that $(Z,\G,\theta)$ is semi-positive precisely when  
$\beta(\det(T_W))\geq \beta(\det(W\times E))$ for all effective $\beta$.

Let $\WmodG$ admit an action by the torus $\T$ with isolated fixed points.
Except in very special cases, $\T$ will not act on $Z/\!\!/\G$.
Nevertheless, as long as we restrict to insertions pulled-back from the cohomology of $\WmodG$, one can still express its quasimap invariants as $\T$-equivariant integrals on the moduli spaces with target $\WmodG$ by the following construction (in Gromov-Witten theory, this construction is due to Kontsevich).

Assume that  for every quasimap $(C,P,u)$ with $C$ a rational curve we have $H^1(C, P\times_G E)=0$. This happens for example when $W\times E$ is generated by $\G$-invariant global sections over the stable locus (so that $\underline{E}$ on $\WmodG$ is also globally generated), see Proposition 6.2.3 and Remark 6.2.4 in \cite{CKM}. 
Following the terminology in Gromov-Witten theory, we will call such a representation $E$ {\it convex}.

On either the moduli space $\mathrm{Q}^{\ke}_{0,k}(\WmodG,\beta)$, or on the graph space $\QGraphe$, consider the universal curve $\cC$, with projection $\pi$, the
universal principal $\G$-bundle $\fP$ on the universal curve, and the universal section $u:\cC\lra \fP\times_\G W$. We have an induced vector bundle 
$$E^\ke_{0,k,\beta}:=u^*(\fP\times_\G (W\times E))$$
on $\cC$. The $H^1$-vanishing assumption implies that $R^1\pi_*E^\ke_{0,k,\beta}=0$ and $R^0\pi_*E^\ke_{0,k,\beta}$ is a vector bundle
on $\mathrm{Q}^{\ke}_{0,k}(\WmodG,\beta)$, or on $\QGraphe$ ($\T$-equivariant, if we choose a $\T$-linearization
of $W\times E$). Its fiber over a quasimap $(C,P,u)$ is the space of sections $H^0(C, P\times_G E)$.

The $\ke$-quasimap invariants
twisted by $E$ are defined by inserting the $\T$-equivariant Euler class of $R^0\pi_*E^\ke_{0,k,\beta}$ in the integrals:
\begin{equation}\label{twisted}\langle \sigma_1\psi_1^{a_1},\dots,\sigma_k\psi_k^{a_k}
\rangle^{\ke, E}_{0,k,\beta}=\int_{[\mathrm{Q}^{\ke}_{0,k}(\WmodG,\beta)]^{\mathrm{vir}}}{\mathrm{e}}(R^0\pi_*E_{0,k,\beta})\prod_{i=1}^k\psi_i^{a_i}ev_i^*(\sigma_i).
\end{equation}
We define similarly the graph space twisted invariants
\begin{equation}\label{twisted graph}
\langle \sigma_1,\dots,\sigma_k\rangle^{QG^\ke, E}_{0,k,\beta}=\int_{[\QGraphe]^{\mathrm{vir}}} {\mathrm{e}}(R^0\pi_*E_{0,k,\beta})\prod_{i=1}^k \tilde{ev}_i^*(\sigma_i).
\end{equation}

Let $j:Z/\!\!/\G\lra\WmodG$ be the inclusion and let $j_{k,\beta}$ be the induced map on quasimap moduli spaces. By Proposition 6.2.2 of \cite{CKM} 
$$(j_{k,\beta})_*([\mathrm{Q}^{\ke}_{0,k}(Z/\!\!/\G,\beta)]^{\mathrm{vir}})=[\mathrm{Q}^{\ke}_{0,k}(\WmodG,\beta)]^{\mathrm{vir}}\cap{\mathrm{e}}(R^0\pi_*E_{0,k,\beta})
$$
Hence the twisted invariant \eqref{twisted} is
equal to the quasimap invariant
$$\langle (j^*\sigma_1)\psi_1^{a_1},\dots,(j^*\sigma_k)\psi_k^{a_k}
\rangle^{\ke}_{0,k,\beta}$$ of the subvariety $Z/\!\!/\G$. 
The same argument gives the analogous equality between $E$-twisted graph space invariants of $\WmodG$ and graph space invariants of $Z/\!\!/\G$.
Notice that the degree $\beta\in \Hom(\Pic^\G(W),\ZZ)$ is measured in $\WmodG$. In many (but not all) examples this will not make a difference, as we will have $\Pic^\G(Z)=\Pic^\G(W)$. In general there is an induced map $j_*:\Hom(\Pic^\G(Z),\ZZ)\lra\Hom(\Pic^\G(W),\ZZ)$ and 
$$\langle (j^*\sigma_1)\psi_1^{a_1},\dots,(j^*\sigma_k)\psi_k^{a_k}
\rangle^{\ke}_{0,k,\beta} = \sum_{j_*(\beta')=\beta}\langle (j^*\sigma_1)\psi_1^{a_1},\dots,(j^*\sigma_k)\psi_k^{a_k}
\rangle^{\ke}_{0,k,\beta'}.$$

We now form the (equivariant) twisted generating series $\tilde{J}^{\ke, E}({\bf t},z)$, $\tilde{S}_{\bf t}^{\ke, E}(z)(\gamma)$, $\tilde{P}^{\ke, E}({\bf t},z)$ by replacing the quasimap invariants 
by their twisted counterparts. Equivalently, we replace in the formulas defining the series the virtual classes $[\mathrm{Q}^{\ke}_{0,k}(\WmodG,\beta)]^{\mathrm{vir}}$ and
$[\QGraphe]^{\mathrm{vir}}$ by their cap products with ${\mathrm{e}}(R^0\pi_*E_{0,k,\beta})$. The intersection pairing is also modified and now reads
$$\lan \sigma_1,\sigma_2\ran^E=\int_{\WmodG}\sigma_1\sigma_2\mathrm{e}(\underline{E}),$$
so that it coincides with the intersection pairing of $j^*\sigma_1$ and $j^*\sigma_2$ on $Z/\!\!/\G$.

For example,
\begin{align*}&\tilde{J}^{\ke, E}({\bf t}, z)=\left(\one+\frac{{\bf t}}{z}\right)\mathrm{e}(\underline{E})+
\sum_{(k,\beta)\neq (0,0),(1,0)}q^\beta\times \\
&\times(\eb)_*\left(\frac{\prod_{i=1}^k ev_i^*({\bf t})}{k!}\cap\mathrm{Res}_{F_0}([\QGraphe]^{\mathrm{vir}}\cap
{\mathrm{e}}(R^0\pi_*E_{0,k,\beta}))\right),
\end{align*}
and 
\begin{align*}
&\tilde{S}^{\ke,E}(z)(\gamma)=\gamma\mathrm{e}(\underline{E})+
\sum_{(m,\beta)\neq (0,0)}\frac{q^\beta}{m!} \times \\
&\times (ev_1)_*\left(\frac{[\mathrm{Q}^\ke_{0,2+m}(\WmodG,\beta)]^{\mathrm{vir}}\cap
{\mathrm{e}}(R^0\pi_*E_{0,k,\beta})}{z-\psi}ev_2^*(\gamma)\prod_{j=3}^{2+m}ev_j^*({\bf t})\right).
\end{align*}

If the insertions $\gamma, {\bf t}$ are taken in $H^*_{\T}(\WmodG,\QQ[[q]])$, then by the above discussion the non-equivariant limits of these
series will give the push-forward by the inclusion $j$ of the
corresponding objects for $Z/\!\!/\G$ restricted to $j^*H^*(\WmodG)$, and with the Novikov parameters specialized if needed: $q^{\beta'}=q^\beta$
for $j_*(\beta')=\beta$.

Now define the twisted $J^\ke$-function, the twisted $S$-operator, and the twisted $P^\ke$-series by requiring
\begin{align}
\lan \tilde{J}^{\ke, E}({\bf t}, z), \delta\ran&=\lan J^{\ke, E}({\bf t}, z) ,\delta\ran^E,\\ \lan \tilde{S}^{\ke,E}(z)(\gamma), \delta\ran&=\lan S^{\ke,E}(z)(\gamma) ,\delta\ran^E,\\
\lan \tilde{P}^\ke,\delta\ran&=\lan P^\ke ,\delta\ran^E,
\end{align}
for every $\delta\in H^*_{\T,\mathrm{loc}}(\WmodG,\QQ)$. Since the tilde-series are divisible by $\mathrm{e}(\underline{E})$, we see that $J^{\ke, E}$, $S^{\ke,E}(z)(\gamma)$,
and $P^\ke$ will be elements in
$H^*_{\T}(\WmodG,\QQ[[q]])\{\!\{z,z^-{1}\}\!\}$ when the insertions are taken in the non-localized equivariant cohomology $H^*_{\T}(\WmodG,\QQ[[q]])$.

Note in addition that the same procedure gives the twisted version of the series \eqref{gen string} defining the generalized string transformation.
 
\subsubsection{Local targets and twisting}
The quasimap theory of the local targets of Example \ref{local targets} can be viewed alternatively as the theory of $\WmodG$ twisted by
the {\it inverse} Euler class ${\mathrm e}^{-1}$. However, if $\T$ acts with isolated fixed points on $\WmodG$, then $\T\times\bS$ will act with
the same fixed points on the total space of the vector bundle $\underline{E}$ (and the $1$-dimensional orbits will also be the same). Hence we may (and will) treat this case
as the untwisted theory of a target with good torus action.

\subsection{The main theorems }  Let $\WmodG$ be acted by $\T$ with isolated fixed points. Let $E$ be a convex $\G$-representation 
as in \S\ref{twisted 1}. 
The statements of the results in this subsection, and the proofs we give in the rest of the section  apply both to the untwisted and the $E$-twisted theories of $\WmodG$. To simplify 
the exposition we will drop $E$ from the notation from now on. Hence, for the rest of this section, the notation $J^\ke, S^\ke, P^\ke, \tau^{\ke_1,\ke_2}...$ will mean either the
generating series for invariants of $\WmodG$, or the series for $E$-twisted invariants.

\begin{Thm}\label{equiv Thm1}  Let $0+\leq\ke_1<\ke_2\leq \infty$ be stability parameters. Let 
$\gamma\in H^*_{{\bf T},{\mathrm{loc}}}(\WmodG,\Lambda)$ be invertible, of the form $\one+O(q)$. Let $\tau_\gamma^{\ke_1,\ke_2}({\bf t})$ be the
($\T$-equivariant) generalized string transformation \eqref{gen string}.  Then
$$S^{\ke_1}_{\tau^{\ke_1,\ke_2}_\gamma ({\bf t})}(z)(\gamma)=S^{\ke_2}_{\bf t}(z)(\gamma).
$$
\end{Thm}

\begin{Cor} The equivariant version of Conjecture \ref{mirror J1} holds for both the $E$-twisted and untwisted theories of $\WmodG$. The equivariant version of
Conjecture \ref{mirror nef} holds for semi-positive $(W,\G,\theta)$, and when $\beta(\det(T_W)) -\beta(\det(W\times E))$ is nonnegative for all $L_\theta$-effective classes $\beta$, it holds for the $E$-twisted theory as well. 
\end{Cor}

Assuming that $\WmodG$ is projective 
and $\gamma\in H^*_{\T}(\WmodG,\QQ[[q]])$ is a non-localized equivariant class, we may pass to the non-equivariant limit in Theorem \ref{equiv Thm1}.

\begin{Cor}\label{nonequiv limit} Let $\WmodG$ be projective, admitting an action by a torus with isolated fixed points. Then

$(i)$ Conjecture \ref{conj S} holds for $\WmodG$. In particular, Conjecture \ref{mirror J1} holds and  
Conjecture \ref{mirror nef} holds for $\WmodG$ with $(W,\G,\theta)$ semi-positive.

$(ii)$ If $j:Z/\!\!/\G\hookrightarrow\WmodG$ is the zero locus of a section of a vector bundle $\underline{E}$ induced by a convex representation, then Conjectures \ref{conj S}
and \ref{mirror J1} hold for the quasimap theory of $Z/\!\!/\G$
restricted to $j^*H^*(\WmodG)$, and with the Novikov parameters specialized by setting $q^{\beta'}=q^\beta$
for $j_*(\beta')=\beta$. If $(Z,\G,\theta)$ is semi-positive, then Conjecture \ref{mirror J1} holds with the same restriction for insertions and the same Novikov specialization.

\end{Cor}

For the proof of next result we will need that the $1$-dimensional $\T$-orbits in $\WmodG$ are also well-behaved.

Let $P^{\infty,\ke}({\bf t},z)$ and $\tau^{\infty,\ke}({\bf t})$ be the $\T$-equivariant series from Lemma \ref{algebraic transformation}. 
As mentioned already, we will use the same notation
for the corresponding objects in the $E$-twisted theory and the statement in the following Theorem includes both the untwisted and twisted cases.

\begin{Thm}\label{equiv Thm2} If the $\T$-action on $\WmodG$ has isolated fixed points and isolated $1$-dimensional orbits, then for all $\ke\geq 0+$, 
$$S^\ke_{\bf t}(z)(P^\ke({\bf t}, z))=S^\infty_{\tau^{\infty,\ke}({\bf t})}(z)(P^{\infty,\ke}(\tau^{\infty,\ke}({\bf t}),z)).$$
In particular, the function $J^\ke({\bf t},z)=S^\ke_{\bf t}(z)(P^\ke({\bf t}, z))$ is on the Lagrangian cone of the ($E$-twisted) Gromov-Witten theory of $\WmodG$.
\end{Thm}

The construction of $P^{\infty,\ke}({\bf t},z)$ and $\tau^{\infty,\ke}({\bf t})$ in Lemma \ref{algebraic transformation} 
does not involve any inversion, hence for ${\bf t}\in H^*_{\T}(\WmodG,\QQ)$ we may pass
to the non-equivariant limit in Theorem \ref{equiv Thm2} to obtain

\begin{Cor} Let $\WmodG$ be projective, admitting an action by a torus with isolated fixed points and isolated $1$-dimensional orbits.  Then

$(i)$ Conjecture \ref{big J} holds for $\WmodG$.

$(ii)$ If $j:Z/\!\!/\G\hookrightarrow\WmodG$ is the zero locus of a section of a vector bundle $\underline{E}$ induced by a convex representation, then Conjecture
\ref{big J} holds for the quasimap theory of $Z/\!\!/\G$
restricted to $j^*H^*(\WmodG)$, with the Novikov parameters specialized by setting $q^{\beta'}=q^\beta$
for $j_*(\beta')=\beta$.

\end{Cor}

The proofs of Theorems \ref{equiv Thm1} and \ref{equiv Thm2} use virtual localization for the $\T$-action on moduli spaces of quasimaps
and will occupy the rest of this section. 

\subsection{Fixed loci}
\subsubsection{$\T$-fixed quasimaps} Fix $\ke>0$, an effective class $\beta\neq 0$ and an integer $k\geq 0$. 
A $\T$-fixed quasimap
$$((C,x_1,x_2\dots, x_{m+2}),P,u)
$$
in $\mathrm{Q}^{\ke}_{0,2+m}(\WmodG,\beta)$ must have all its nodes, markings, ramifications, and base-points lying over the fixed locus $(\WmodG)^{\T}$.
Explicitly, this means the following. For any irreducible component $C'\cong\PP^1$ of $C$, let $(P',u')$ be the restriction of the pair $(P,u)$ to $C'$.
We have a rational map $$[u']:C'--\ra\WmodG,$$ inducing a regular map $$[u']_{reg}:C'\lra\WmodG$$ (possibly of smaller degree than the degree
of the quasimap $(C', P', u')$). Then one of the following must hold:
\begin{enumerate}
\item $[u']_{reg}$ is a constant map with image a fixed point $\mu\in (\WmodG)^{\T}$. We will say in this case that $C'$ is a {\it contracted component} of the quasimap.
\item $[u']_{reg}=[u']$ (i.e., there are no base-points on $C'$) and the map is a cover of a $1$-dimensional $\T$-orbit in $\WmodG$, totally ramified over the
two fixed points in the orbit.
The nodes/markings on $C'$, if any, must be mapped to the fixed points as well.
\item $[u']_{reg}\neq [u']$ and is a ramified cover of an orbit as in $(2)$ above, while $[u']$ has a base-point at one of the ramification points and a node at the other
ramification point. This possibility may occur only when $\beta(L_\theta)<\ke\leq 1$.
\end{enumerate}

\subsubsection{Components of the fixed loci}
When the $\T$-action on $\WmodG$ satisfies the additional condition that the $1$-dimensional $\T$-orbits
are isolated, the description of the connected components of the $\T$-fixed loci in the stable map moduli spaces $\mathrm{Q}^{\ke>1}_{0,2+m}(\WmodG,\beta)$ 
goes back to Kontsevich and is well-known. They are indexed by certain decorated trees and each such component is isomorphic to (a finite quotient
of) a product of moduli spaces of stable pointed curves $\overline{M}_{0,n}$, with the factors corresponding to the vertices of the tree.
For general $\ke$ a similar description holds, except
that each vertex factor now is a $\T$-fixed component in a $\ke$-quasimap space with target a fixed point $\mu\in (\WmodG)^{\T}$, viewed as a GIT quotient by $\G$.
In special cases, e.g. when $\WmodG$ is a Grassmannian, or a toric variety, they turn out to be
certain Hassett
moduli spaces $\overline{M}^\ke_{0,k|d}$ of stable marked curves with weights, cf. \cite{MOP}, \cite{Toda}, \cite{CKtoric2}.

When $1$-dimensional orbits vary in families, the fixed loci can still be indexed by decorated trees, but
their precise structure is more complicated, since in this case there will be factors of positive dimension
corresponding not just to the vertices
of the trees, but also to the edges. These edge moduli spaces will depend on the target and may be difficult to identify explicitly. 
An example is worked out by
Okounkov and Pandharipande, who treated the case of stable maps to the Hilbert scheme of points in $\CC^2$ in \S3.8.2 - 3.8.3 of \cite{OP}. 
However, we will only need the coarser (and obvious) information that certain edge factors are {\it independent on $\ke$} 
for a given target $\WmodG$, rather than a full description of 
either the edge or vertex moduli spaces.

\subsubsection{Unbroken quasimaps and unbroken components}
Each marking $x$ of $C$ has a fractional cotangent weight $\alpha_x$ given by the $\T$-representation induced by
the quasimap on the cotangent space $T^*_xC$. If $x$ is on a contracted component, then the weight is zero. 
Otherwise, $-n\alpha_x=w$ for some integer $n\geq 1$  and a weight $w$ of the 
representation $T_\mu (\WmodG)$, where $\mu\in(\WmodG)^{\T}$ is the image of $x$. Note that the weights of  $T_\mu (\WmodG)$ are all non-zero since the fixed points are
isolated.

Similarly, if $x$ is a node of $C$ and $C', C''$ are the components of $C$ incident to $x$, there
are fractional cotangent weights $\alpha_{C',x}$ and $\alpha_{C'',x}$ from the $\T$-representations $T^*_xC'$ and $T^*_xC''$ respectively. The node $x$ is called a 
breaking node if
$$\alpha_{C',x}+\alpha_{C'',x}\neq 0.
$$
As in \cite{OP}, we call a $\T$-fixed $\ke$-stable 
quasimap of class $\beta\neq 0$ to $\WmodG$ {\it unbroken} if all nodes of its domain curve are non-breaking. This implies that either the entire domain
curve is contracted by the quasimap, or that no components are contracted. We discuss the latter situation. Since we consider only quasimaps with at least two markings, the number of markings must be exactly
equal to two and there can be no base points. Hence a (non-contracting) unbroken quasimap must be a two-pointed stable map $f:(C,x_1,x_2)\lra\WmodG$. 
The domain curve $$C=C_1\cup\dots\cup C_l$$ is a chain of $\PP^1$'s with $x_1\in C_1$, $x_2\in C_l$ and nodes $$y_j=C_j\cap C_{j+1},\; j=1,\dots, l-1.$$
The weights at markings and nodes are all nonzero and satisfy
$$\alpha_{x_1}=-\alpha_{C_1,y_1},\;\; \alpha_{C_j,y_j}=-\alpha_{C_{j+1},y_j},\;\; \alpha_{C_l,y_{l-1}}=-\alpha_{x_2},$$
and in particular $\alpha_{x_1}=-\alpha_{x_2}$. We say that the stable map connects the fixed points $\mu=f(x_1)$ and 
$\nu=f(x_2)$. 

If $\cO(\theta)$ is given a nontrivial $\T$-linearization, let $w_\mu$ and $w_\nu$ be the weights of the $\T$-action on the fibers over $\mu$ and $\nu$ respectively.
Then a localization computation as in Lemma 4 of \cite{OP} shows that
$$\alpha_{x_1}=\frac{-w_\mu+w_\nu}{\beta(L_\theta)}.
$$
By the non-vanishing of $\alpha_{x_1}$ we conclude also that there are no non-contracting unbroken quasimaps connecting a $\T$-fixed point to itself.

Let $M$ be a connected component of the $\T$-fixed locus in the moduli space $\mathrm{Q}^{\ke}_{0,2}(\WmodG,\beta)$. 
If $M$ contains a non-contracting unbroken quasimap with
starting weight $\alpha_{x_1}$ and joining two fixed points $\mu$ and $\nu$, then all quasimaps in $M$
are unbroken, with the same starting weight, and join the points $\mu$ and $\nu$. 
(Note, however that the number of irreducible components of the domain curve may vary within $M$.) The important fact for us will be that
these unbroken components are the same for all
values of $\ke$, since they always parametrize stable {\it maps} of class $\beta$.

When there are more than two markings, the unbroken components will appear as the $\ke$-independent edge moduli spaces that were mentioned earlier. 
(There may also be edge factors corresponding to unbroken quasimaps with only one marking, but we will not be concerned with these explicitly.)

\subsection{Recursion}
We divide connected components $M$ of the fixed locus 
$\mathrm{Q}^{\ke}_{0,2+m}(\WmodG,\beta)^{\T}$ into two types. 
We say that $M$ is of {\it initial type} if it parametrizes quasimaps for which the {\it first} marking $x_1$ is on a contracted
irreducible component of the domain curve. We say $M$ is of {\it recursion type} if it parametrizes quasimaps with the first marking on a non-contracted
irreducible component.
Every $M$ is either of initial type, or of recursion type.

Every recursion component is of the form 
\begin{equation}\label{recursion splitting}
M\cong M' \times _{(\WmodG)^{\T}} M'' ,\end{equation} 
where

$(i)$  $M'\subset \overline{M}_{0,2}(\WmodG,\beta_{M'})^{\T}$ is an unbroken component for some $\beta_{M'}\neq 0$.

$(ii)$ $M''$ is a connected component of
$\mathrm{Q}^{\ke}_{0,2+m}(\WmodG,\beta-\beta_{M'})^{\T}$.

$(iii)$ $\ka_M=\ka _{M'} \neq \ka _{M''}$, where $\ka _M$ denotes the
cotangent weight at the first marking in any element of $M$ (and similarly for $M', M''$). 

\noindent The fiber product is taken using the evaluation map at the second marking in $M'$ and the
evaluation map at the first marking in $M''$. The component $M''$ may be of either initial or recursion type. 
In the unstable case $m=0$ and $\beta_{M'}=\beta$ we take $M''$ to be a fixed point in $\WmodG$ and $\ka_{M''}=0$.

We conclude  \[ \mathrm{Q}^{\ke}_{0,2+m}(\WmodG,\beta)^{\T} = 
\coprod _{M \text{ initial} } M \coprod _{M' \text{ unbroken}} \left( M' \times _{(\WmodG)^{\T}}\coprod _{ \ka _{M'} \ne \ka _{M''}} M'' \right)\]

Let 
$$\{ \delta_\mu \; | \; \mu\in (\WmodG)^{\T}\}$$
be the basis of the localized cohomology $H^*_{\T,\mathrm{loc}}(\WmodG,\QQ)$ consisting of the equivariant fundamental classes of the fixed points, i.e.,
$ \delta_\mu$ is equal to (the dual of) $(i_\mu)_*[\mu]$, where $ i_\mu:\{\mu\}\hookrightarrow\WmodG$ is the embedding.

For a cohomology class $\gamma\in H^*_{\T}(\WmodG,\Lambda)$ of the form $\gamma=\gamma_0 +O(q)$ 
we consider the components
 \[ S_\mu^\ke (\gamma) := \lan S_{\bf t}^\ke(z)(\gamma) , \delta _\mu \ran = i_\mu^*(\gamma)+ \sum_{(\beta,m)\neq (0,0)} 
 \frac{q^\beta}{m!} \lan \frac{\delta _\mu}{z-\psi}, \gamma,{\bf t},\dots,{\bf t} \ran _{0,2+m,\beta}^\ke\]
 of $S_{\bf t}^\ke(z)(\gamma)$ in the fixed point basis. Alternatively, these are the restrictions of $S_{\bf t}^\ke(z)(\gamma)$ at the fixed points,
 $$S_\mu^\ke (\gamma) = i_\mu^*(S_{\bf t}^\ke(z)(\gamma)).
 $$
 
 If, as in \S\ref{J}, $\{ t_i\}$ are the coordinates of ${\bf t}$ in a basis of $H^*_{\T,\mathrm{loc}}(\WmodG)$ over $K=\QQ(\lambda_j)$, then $S_\mu^\ke (\gamma)$ is an
 element of $\Lambda[[t_i,\frac{1}{z}]]=K[[q,t_i,\frac{1}{z}]]$.
 
 Given a $\T$-fixed point $\mu$, a class $\beta$ and an integer $m\geq 0$, we partition the set of connected components of 
 $\mathrm{Q}^{\ke}_{0,2+m}(\WmodG,\beta)^{\T}$ into three disjoint subsets
 \begin{equation}\label{partition}
 V(\mu,\beta,m)\coprod In(\mu,\beta,m)\coprod Rec(\mu,\beta, m)
 \end{equation}
 as follows:
 \begin{itemize}
 \item $V(\mu,\beta,m)$ consists of all components for which the first marking does {\it not} lie over $\mu$.
 \item $In(\mu,\beta,m)$ consists of all components of initial type for which the first marking lies over $\mu$.
 \item $Rec(\mu,\beta, m)$ consists of all components of recursion type for which the first marking lies over $\mu$.
 \end{itemize}
 
 \begin{Lemma}\label{poles} For each fixed $\beta$ and each tuple $\underline{k}:=(k_i)_i$ of nonnegative integers, the coefficient 
 of the monomial $q^\beta\prod_i t_i^{k_i}$ in $S_\mu^\ke (\gamma)$
 is a rational function of $z$ with coefficients in $K$. This rational function decomposes as a sum of simple partial fractions with denominators either powers of $z$, or
 powers of linear factors $z-\ka$, where $-n\ka$ is one of weights of the $\T$-representation  $T_\mu (\WmodG)$ for 
some $n \in \ZZ _{>0}$
 
 \end{Lemma}
 
 \begin{proof}It suffices to assume $\gamma=\gamma_0\in H^*_{\T}(\WmodG,\QQ)$. 
 Let $a_{\beta,\underline{k}}$ be the coefficient in the Lemma and let $m=\sum_ik_i$ be the total degree in the $t_i$'s. Then
 \begin{align*}
 a_{\beta,\underline{k}}&=\int_{[\mathrm{Q}^{\ke}_{0,2+m}(\WmodG,\beta)]^{\mathrm{vir}}} \frac{ev_1^*(\delta_\mu)c_{\beta,\underline{k}}}{z-\psi}\\
 &=\sum_{j=0}^\infty \int_{[\mathrm{Q}^{\ke}_{0,2+m}(\WmodG,\beta)]^{\mathrm{vir}}} \frac{\psi^j ev_1^*(\delta_\mu)c_{\beta,\underline{k}}}{z^{j+1}}
 \end{align*}
 for some cohomology class $c_{\beta,\underline{k}}$ pulled-back from $(\WmodG)^{m+1}$ by the product of the remaining evaluation maps. In the case of
 $E$-twisted theory the Euler class of the twisting bundle is also incorporated in $c_{\beta,\underline{k}}$. By the formula above,
 $ a_{\beta,\underline{k}}$ is apriori an element in $K[[1/z]]$.
 
 We apply virtual localization to
 the integrals in the above expression. Each integral is a sum of contributions from the components of the fixed point locus. We use the partition \eqref{partition} to analyze
 these contributions.
 
 First, the contribution from components in $V(\mu,\beta,m)$ vanishes, due to the factor $ev_1^*(\delta_\mu)$ in the integrand. 
 
 Next, each initial component $M$ in $In(\mu,\beta,m)$ contributes only fractions with denominator a power of $z$. 
 This is because $\psi$ restricts to a {\it non-equivariant}
 class to $M$, which is therefore nilpotent.
 
 Finally, if $M$ is a recursion component in $Rec(\mu,\beta, m)$, then the restriction of $\psi$ to $M$ is $\psi_0+\ka_M$, with $\psi_0$ a non-equivariant class. 
 If $M$ is positive-dimensional the contribution can be written as
 $$
 \sum_{j=0}^\infty\int_{[M]^{\mathrm{vir}}} \frac{\psi_0^j(ev_1^*(\delta_\mu)c_{\beta,\underline{k}})|_M}{(z-\ka_M)^{j+1}\mathrm{e}(N^{\mathrm{vir}})}.
 $$
Again, by nilpotency 
of $\psi_0$, $M$ contributes finitely many fractions with denominator powers of $z-\ka_M$. If $M$ has dimension zero, then $\psi$ restricts to the pure weight $\ka_M$, so the
series 
$$ \sum_{j=0}^\infty\int_{[M]^{\mathrm{vir}}} \frac{\ka_M^j(ev_1^*(\delta_\mu)c_{\beta,\underline{k}})|_M}{z^{j+1}\mathrm{e}(N^{\mathrm{vir}})}
$$
is a geometric series and therefore the contribution is a fraction with a simple pole at $z=\ka_M$. The Lemma is proved.
 \end{proof}

With a more careful analysis of the contributions from the recursion-type fixed loci, the virtual localization theorem shows in fact that the restrictions to fixed points
$S_\mu^\ke$ satisfy a certain recursion relation. 

For a given $\mu\in(\WmodG)^{\T}$, let $U(\mu)$ denote the set of all components $M'\subset\overline{M}_{0,2}(\WmodG, \beta _{M'} )^{\T} $ 
(for varying $\beta_{M'}\neq 0$) parametrizing {\it unbroken}
stable maps $(C,x_1,x_2,f)$ with $f(x_1)=\mu$. For such an $M'$ we set $\nu_{M'}=f(x_2)$.

\begin{Lemma} \label{recursion lemma}
 $S_\mu^\ke$ satisfies the recursion relation
       \begin{equation}\label{recursion formula} 
       S_\mu ^\ke (z) = R_\mu^\ke (z) + \sum _{ M'\in U(\mu)} q^{\beta_{M'}} 
\left\lan \frac{\delta _\mu}{z - \ka _{M'}- \psi _0 }, S_{\nu_{M'}, \ka _{M'}}^\ke |_{z=\ka _{M'}-\psi _{\infty}}
\right\ran _{M'}
       \end{equation} 
where: 
\begin{itemize}
\item $\psi _0$, $\psi _\infty$ are nonequivariant cohomology classes.
\item Each $(q,\{t_i\})$-coefficient of $R_\mu ^\ke (z)$ is an element in $K[1/z]$.
\item For $\nu\in (\WmodG)^{\T}$, $S_{\nu, \ka}^\ke (z)$ is $S_\nu^\ke (z)$ after removing the partial fraction terms with poles at $z=\ka$ (this is well defined by Lemma
\ref{poles}).
\item The subscript $M'$ for the bracket means the $\T$-virtual localization contribution of the component $M'$ to the intersection 
number on $\overline{M}_{0,2}(\WmodG, \beta _{M'} )$.
\end{itemize}

\end{Lemma}

 \begin{proof}

We apply virtual $\T$-localization to $S^\ke_\mu$. The terms with $\beta=0$ and varying $m$ sum to $ i_\mu^*(e^{{\bf t}/z}\gamma)$. For each $\beta\neq 0$ and each $m\geq 0$,
the components of the fixed locus in $V(\mu,\beta,m)$ contribute zero, while the components of initial type in $In(\mu,\beta,m)$ contribute polynomials of
$1/z$ with coefficients (homogeneous of degree $m$) in $K[\{t_i\}]$. Summing over all $m,\beta$ gives the initial terms $R_\mu^\ke (z)$ as described in the Lemma.
The remaining part is obtained by summing over all $m$ and $\beta$ the contributions of components of recursion type in $Rec(\mu,\beta,m)$.

 The contribution of an $M\in Rec(\mu,\beta,m)$ 
 is 
 \begin{equation}\label{rec cont}q^{\beta} \int _{[M]^{\mathrm {vir}}} \frac{1}{{\mathrm {e}}^{\T} (N^{\mathrm {vir}}_M)}
 \left(\left.\frac{ev_1^*(\delta_\mu)ev_2^*(\gamma)\prod_{i=1}^m ev_{2+i}^*({\bf t})}{(z-\psi)m!}\right|_M\right) .
 \end{equation}
By \eqref{recursion splitting}, $$M\cong M' \times _{(\WmodG)^{\T}} M'' ,$$ with $M'\in U(\mu)$ parametrizing two-pointed maps of some class $\beta_{M'}\neq 0$
and $M''\in In(\nu_{M'},\beta-\beta_{M'},m)\coprod Rec(\nu_{M'},\beta-\beta_{M'},m)$.
Moreover, $\alpha_{M'}=\alpha_{M}$, and $\alpha_{M''}\neq\alpha_{M}$.

Recall that the virtual class of the fixed locus $M$, respectively the virtual normal bundle, are obtained from the fixed part, respectively the moving part of the {\it absolute}
obstruction theory of  $\mathrm{Q}^{\ke}_{0,2+m}(\WmodG,\beta)$. In turn, the absolute obstruction theory is obtained 
via distinguished triangles from the complex \eqref{obs theory}
$$\left (R^\bullet\pi_*(u^*\mathbb{R}T_\varrho)\right )^\vee,$$
the relative cotangent complex ${\mathbb {L}}^\bullet _{\BunG/\fM_{0,m+2}}$, and the cotangent complex ${\mathbb {L}}^\bullet_{\fM_{0,m+2}}$, see Remark 
4.5.3 in \cite{CKM}. 
The relative obstruction theory $E^\bullet$ over $\fM_{0,m+2}$ is obtained from the distinguished triangle
$$E^\bullet\lra  (R^\bullet\pi_*(u^*\mathbb{R}T_\varrho))^\vee\lra g^*{\mathbb {L}}^\bullet _{\BunG/\fM_{0,m+2}}[1]
$$
where $g$ is the projection to $\BunG$.

Consider the exact sequence
\begin{equation}\label{normalization}
0\lra \cO_C \lra\cO_{C'}\oplus\cO_{C''}\lra\cO_{C'\cap C''}\lra 0
\end{equation}
coming from splitting the domain curve at the node where the unbroken map $f:C'\lra \WmodG$ meets the rest of the quasimap.
The complex $E^\bullet$ on $M$ differs from the sum of the corresponding complexes on $M'$ and $M''$ only by a complex supported at the node, whose
cohomology is the representation $T_{\nu_{M'}}(\WmodG)$ in degree zero. Since all its $\T$-weights, are nontrivial, this term will only appear in the difference between
the virtual normal bundles. The complex ${\mathbb {L}}^\bullet_{\fM_{0,m+2}}$ differs from the sum of the complexes on the factors only by the deformation space at the
node which is the product
$$T_{x_2}C'\boxtimes T_{x_1}C'',$$
of tangent spaces at the indicated markings. It has weight $\alpha_M-\alpha_{M''}\neq 0$, so will contribute again to the normal bundle difference.

We conclude that 
$$[M]^{\mathrm {vir}}=[M']^{\mathrm {vir}}\times [M'']^{\mathrm {vir}}$$
and that 
$$\frac{1}{{\mathrm {e}}^{\T} (N^{\mathrm {vir}}_M)}= 
\frac{{\mathrm {e}}^{\T} ([T_{\nu_{M'}}(\WmodG)])}{{\mathrm {e}}^{\T} (N^{\mathrm {vir}}_{M'}){\mathrm {e}}^{\T} (N^{\mathrm {vir}}_{M''})(-\psi_2^{M'}-\psi_1^{M''})}.$$
Furthermore, the cotangent weight at the second marking on $M'$ is $-\alpha_M=-\alpha_{M'}$, hence  $-\psi_2^{M'}=\alpha_{M}-\psi_\infty$ with $\psi_\infty$ non-equivariant.
Hence the contribution \eqref{rec cont} can be written as a product of two virtual localization integrals on the factors,
\begin{align*}
\left(q^{\beta_{M'}}\int _{[M']^{\mathrm {vir}}} 
\frac{ev_1^*(\delta_\mu)}{(z-\alpha_{M'}-\psi_0){\mathrm {e}}^{\T} (N^{\mathrm {vir}}_{M'})}\right) \times \\
\left( q^{\beta-\beta_{M'}}
\int _{[M'']^{\mathrm {vir}}}\frac{ev_1^*(\delta_{\nu_{M'}}) ev_2^*(\gamma)\prod_{i=1}^m ev_{2+i}^*({\bf t})}{(\alpha_{M'}-\psi_{\infty}-\psi){\mathrm {e}}^{\T} (N^{\mathrm {vir}}_{M''})}   \right)
\end{align*}
The sum over all recursion components is then easily 
rearranged to obtain the recursion part
of \eqref{recursion formula}. 

Finally, we note that the argument works also in the twisted case, since the exact sequence \eqref{normalization} shows that 
$$H^0(C,P\times_\G E)\oplus \underline{E}_\mu= H^0(C', P|_{C'}\times_\G E)\oplus H^0(C'', P|_{C''}\times_\G E)
$$
as $\T$-representations,
and hence that the Euler class of the twisting bundle in the integrand also factors.
\end{proof}

 \begin{Rmk} \label{recursion remark}
 $(i)$ The main point of Lemma \ref{recursion lemma} is the following: up to initial terms which 
 are polynomial in $1/z$, the $q^\beta$-term in $S^\ke_\mu$ is calculated from various $q^{\beta-\beta' }$-terms 
  with $\beta '\neq 0$ in $S^\ke_\nu$ for other $\nu$'s 
  by a universal formula whose coefficients are {\it independent on} $\ke$ (but may depend on $\gamma$). It is in this sense that we think of \eqref{recursion formula}
  as a recursion relation. Given the initial terms $R_\mu ^\ke (z)$, the set of all $S^\ke_\mu,\; \mu \in(\WmodG)^{\T}$ is uniquely recovered from \eqref{recursion formula}.
 
 $(ii)$ The statement of the Lemma remains valid if we substitute ${\bf t}$ by $\tau({\bf t},q)\in H^*_{\T,\mathrm{loc}}(\WmodG,\Lambda)$.
 
 $(iii)$ If the $1$-dimensional $\T$-orbits in $\WmodG$ are isolated, then each unbroken component is just a point and the classes $\psi_0$ and $\psi_\infty$ vanish.
 The proof of Lemma \ref{poles} shows that in this case the coefficients of $S_\mu^\ke (\gamma)$ have
 only {\it simple} poles at $\frac{w}{-n}$ where $n$ runs over the positive integers and $w$ runs over the tangent weights at $\mu$ (and high order poles at $z=0$).
 Furthermore, since the tangent weights at each fixed point are independent, the formula \eqref{recursion formula} for $\gamma=\one$ turns into the usual Givental recursion
 \begin{equation}\label{isolated recursion}
  S_\mu ^\ke (z) = R_\mu^\ke (z) + \sum _{\nu\in o(\mu)} \sum_{n=1}^{\infty}q^{n\beta(\mu,\nu)}\frac{C_{\mu,\nu,n}}{z+\frac{w(\mu,\nu)}{n}}S_\nu ^\ke (-\frac{w(\mu,\nu)}{n})
 \end{equation}
 obtained in \cite{Givental-equiv}, \cite{Givental} for Gromov-Witten invariants.
 Here the first sum is over the set $o(\mu)$ of all fixed points $\nu$ connected to $\mu$ by a $1$-dimensional $\T$-orbit, $\beta(\mu,\nu)$ is
 the homology class of the orbit, and $w(\mu,\nu)$ is the tangent weight at $\mu$
 corresponding to the orbit. The coefficients $C_{\mu,\nu,n}$ are explicit rational functions in $K=\QQ(\lambda_j)$ and are {\it independent} on $\ke$.
 
 Hence the Lemma generalizes Givental's result in two 
 main directions: to all $\ke$-stable quasimap invariants, and to torus actions with not necessarily isolated $1$-dimensional orbits.

  \end{Rmk}
  
  \subsection{Polynomiality revisited} All the generating
  functions $S^\ke_{\bf t}(z)(\gamma)$ satisfy a polynomiality property analogous to the one given in \S\ref{poly}
  for the $J^\ke$-functions, and proved in a similar manner. However, we will need a slightly more refined localized version for each of the fixed-point components $S^\ke_\mu$.
  
  \begin{Lemma}\label{local poly} For each $\mu\in(\WmodG)^{\T}$ the series
  $$D(S^\ke_\mu):= S^\ke_\mu(q,{\bf t},z)S^\ke_\mu(qe^{-zyL_\theta},{\bf t},-z)
  $$
  has no pole at $z=0$, i.e., it is naturally an element in the formal power series ring
  $K[[q,t_i,y,z]]$.
  \end{Lemma}
  
  \begin{proof} For effective classes $\beta,\beta_1,\beta_2$ consider the morphism
$$\Phi_\ke\circ b_{\ke, \{d(\beta_1),d(\beta_2)\}}\circ\iota_{\ke} : QG^\ke_{0,2+m,\beta}(\WmodG) \lra QG_{0,0,d(\beta)+d(\beta_1)+ d(\beta_2)}(\PP^N_{A^\G})$$
from \S\ref{universal} and denote by $U_{\beta_1,\beta_2}(L_\theta)$ the universal $\CC^*$-equivariant line bundle obtained by pulling back $\cO(1)$ with
the canonical linearization.
  
 Next, for a given $\mu\in(\WmodG)^{\T}$, let
 $$QG^\ke_{0,2+m,\beta}(\WmodG)_\mu$$
 be the $\T$-fixed locus in the $QG^\ke_{0,2+m,\beta}(\WmodG)$ consisting of quasimaps for which the parametrized $\PP^1$ is contracted to the point $\mu$.
 It is a $\CC^*$-invariant substack. Denote by $\kappa$ the inclusion map into the graph space.
 We then put
 $${\mathrm{Res}}_\mu [QG^\ke_{0,2+m,\beta}(\WmodG)]^{\mathrm{vir}}:= \kappa_*
 \left ( \frac{[QG^\ke_{0,2+m,\beta}(\WmodG)_\mu]^{\mathrm{vir}}}{\mathrm{e}^{\T}(N^{\mathrm{vir}})}\right),$$
the $\T$-localization residue of the virtual class $[QG^\ke_{0,2+m,\beta}(\WmodG)]^{\mathrm{vir}}$ of the entire graph space at this locus. (In the case of twisted theory, we replace
as usual $[QG^\ke_{0,2+m,\beta}(\WmodG)_\mu]^{\mathrm{vir}}$ by 
$\kappa^*\mathrm{e}(R^0\pi_*E_{0,2+m,\beta})\cap[QG^\ke_{0,2+m,\beta}(\WmodG)_\mu]^{\mathrm{vir}}$.)
  
Finally,  we write $\gamma$ as a $q$-series
  $$\gamma=\gamma_0 +\sum_{\beta\neq 0}q^\beta \gamma_{\beta}$$
  with $\gamma_\beta\in H^*_{\T,\mathrm{loc}}(\WmodG)$.
  
  Consider the generating series
  
  \begin{align*}
 & \lla \gamma p_0,\gamma p_\infty ; e^{c_1(U(L_\theta))}\rra^{QG^\ke}_{2;\mu} :=\sum _{m,\beta\geq 0}\frac{q^\beta}{m!} 
 \sum_{\beta_1,\beta_2\geq 0} q^{\beta_1}q^{\beta_2}\times \\
& \int_{{\mathrm{Res}}_\mu[QG^\ke_{0,2+m,\beta}(\WmodG)]^{\mathrm{vir}}}e^{c_1(U_{\beta_1,\beta_2}(L_\theta))y}ev_1^*(\gamma_{\beta_1}p_0)
ev_2^*(\gamma_{\beta_2}p_\infty)
  \prod_{i=1}^m ev_{2+i}^*({\bf t}).
  \end{align*}
  It is defined {\it without} $\CC^*$-localization and hence it is naturally a power series in $z$.
  
  On the other hand, we can calculate it by applying virtual localization for the $\CC^*$-action on the graph space, as in the proofs of Lemma \ref{poly lemma} 
  and of Proposition \ref{unitary}. The result is
  $$e^{w_{\mu}y}D(S^\ke_\mu),
  $$
  where $w_\mu=w(\cO(\theta);\mu)$ is the $\T$-weight on the fiber of $\cO(\theta)$ at $\mu$. (In the twisted case, there is an additional factor given by the character
  $\mathrm{e}(\underline{E}_\mu)$ of the $\T$ representation of the fiber of $\underline{E}$ at $\mu$.)
  The Lemma is proved.
  \end{proof}
  
  \begin{Rmk}\label{modification} Since the series
$P^\ke$ and $P^{\infty,\ke}$ appearing in Theorem \ref{equiv Thm2}  contain only nonnegative powers of $z$, it is clear that 
  (after making appropriate adjustments to the universal bundles $U(L_\theta)$) 
  the proof also applies to the components of  $S^\ke_{\bf t}(z)(P^\ke({\bf t}, z)) $ and $S^\infty_{{\bf {\tau}}}(z)(P^{\infty ,\ke}({\bf {\tau}},z))$, where ${\bf{\tau}}={\bf {\tau}}^{\infty,\ke}(q,{\bf t})$.
  \end{Rmk}
  
  \subsection{Uniqueness Lemma} The following statement is an appropriately generalized version of Givental's Uniqueness Lemma, 
  see Proposition 4.5 in \cite{Givental} and Lemma 3
in \cite{Kim}.

\begin{Lemma}  \label{uniqueness lemma}
  For $j=1,2$, let $$\{S_{j,\mu}\; |\; \mu\in (\WmodG)^{\T}\}$$
  be two systems of power series in $\Lambda[[t_i]]\{\!\{z,z^{-1}\}\!\}=K[[q,t_i]]\{\!\{z,z^{-1}\}\!\}$ satisfying the following properties:
  \begin{enumerate}
  \item If we write
$$S_{j,\mu}=\sum_{\beta}q^\beta\sum_{ \underline{k}} a_{j,\mu,\beta,\underline{k}}(z) \prod_i t_i^{k_i}$$
for $j=1,2$, then $a_{j,\mu,\beta,\underline{k}}(z)$ are in $K(z)$ and decompose into sums of partial fractions with specified poles, as given by Lemma \ref{poles}.
\item The systems 
$$\{S_{j,\mu}\; |\; \mu\in (\WmodG)^{\T}\},\;\;\; j=1,2$$
satisfy the recursion relation \eqref{recursion formula}.
\item  For $j=1,2$ and  all $\mu$, the convolution
$$ D(S_{j,\mu}):= S_{j,\mu}(q,{\bf t},z)S_{j,\mu}(qe^{-zyL_\theta},{\bf t},-z)
$$
has no pole at $z=0$.
\item For all $\mu$, $$S_{1,\mu}=S_{2,\mu}\;\;\; \mathrm{mod} \;1/z^2,$$ 
\item For $j=1,2$ and  all $\mu$, $$S_{j,\mu}=i_\mu^*(e^{{\bf t}/z})+O(q).$$
 \end{enumerate}
 Then  $S_{1,\mu}=S_{2,\mu}$ for all $\mu\in (\WmodG)^{\T}$.
  
  \end{Lemma}
  
  \begin{proof}
  For each tuple $\underline{k}:=(k_i)_i$ of nonnegative integers and each effective $\beta$ define the {\it bi-degree}
 of the monomial $q^\beta\prod_i t_i^{k_i}$ to be
 $$\left(\sum_ik_i,\beta(L_\theta)\right)\in \NN\times \NN.$$
(Note the reversal of ordering in the pair: the total degree in the $t$-variables appears first!)

For a pair $(m,d)$ of nonnegative integers, denote by $S^{(m,d)}_{j,\mu}$ the 
part of bi-degree $(m,d)$ of $S_{j,\mu}$,
$$S^{(m,d)}_{j,\mu}:=\sum_{\{\beta : \beta(L_\theta)= d\} }q^\beta\sum_{\{ \underline{k}: \sum k_i= m\}} a_{j,\mu,\beta,\underline{k}}(z) \prod_i t_i^{k_i}.$$
To prove the Lemma it suffices to show that for all $\mu\in (\WmodG)^{\T}$ and all $(m,d)\in \NN\times \NN$ we have 
\begin{equation}\label{induction}
S^{(m,d)}_{1,\mu}=S^{(m,d)}_{2,\mu}. 
\end{equation}
We do so
by induction on $(m,d)$, using the lexicographic order
$$(m',d')< (m,d) \; {\text {iff}}\; m'<m,\; {\text{or}}\; m'=m\;{\text{and}}\; d'<d.$$
The base case $d=0$ and all $m\in\NN$ holds by property $(5)$ above. Let $(m,d)$ be fixed with $d\geq 1$. We assume \eqref{induction} is true for all fixed points
$\mu$ and all $(m',d')< (m,d)$. Consider the difference
$$\Delta:=D(S_{1,\mu})-D(S_{2,\mu})$$
and let $\Delta^{(m,d)}$ be its part of bi-degree $(m,d)$. By the induction assumption
\begin{equation}\label{difference}\Delta^{(m,d)}=(S^{(m,d)}_{1,\mu}-S^{(m,d)}_{2,\mu})+e^{-zyd}(S^{(m,d)}_{1,\mu}(-z)-S^{(m,d)}_{2,\mu}(-z)).
\end{equation}
By the recursion property $(2)$ and the induction assumption again, the difference
$$S^{(m,d)}_{1,\mu}-S^{(m,d)}_{2,\mu}$$
is a polynomial in $1/z$. We write it as
$$z^{-2a}\left(\frac{A}{z}+B+O(z)\right),$$ 
with $a\geq 0$ and $A,B\in\Lambda[\{t_i\}]$ homogeneous polynomials of bi-degree $(m,d)$.
Then \eqref{difference} can be rewritten as
$$\Delta^{(m,d)}=z^{-2a}\left(2B+ydA+O(z)\right).$$
By the polynomiality property $(3)$, $\Delta^{(m,d)}$ has no pole at $z=0$. Hence if $a\geq1$, then we must have $2B+ydA=0$, which in turn implies $A=B=0$. 
By descending induction we obtain that $a=0$. But then $S^{(m,d)}_{1,\mu}=S^{(m,d)}_{2,\mu}$, since they agree modulo $1/z^2$ by property $(4)$.
\end{proof}

\subsection{Proof of Theorem \ref{equiv Thm1}  }
For every fixed point $\mu\in (\WmodG)^{\T}$ put 
$$S_{1,\mu}:= \lan S^{\ke_1}_{\tau^{\ke_1,\ke_2}_\gamma ({\bf t})}(z)(\gamma),\delta_\mu\ran,$$
$$S_{2,\mu}:= \lan S^{\ke_2}_{\bf t}(z)(\gamma),\delta_\mu\ran.$$
Theorem \ref{equiv Thm1} will be proved if we show that
$$S_{1,\mu}=S_{2,\mu},\;\;\; \forall \mu\in (\WmodG)^{\T}.$$
But this follows from the Uniqueness Lemma \ref{uniqueness lemma}, since $S_{j,\mu}$ are power series in $\Lambda[[t_i,1/z]]$
satisfying the five properties listed in that Lemma. Indeed
\begin{itemize}
\item Property $(1)$ holds by Lemma \ref{poles}.
\item Property $(2)$ holds by Lemma \ref{recursion lemma}. 
\item Property $(3)$ holds by Lemma \ref{local poly}.
\item Property $(4)$ holds by \eqref{1/z matching}.
\item Property $(5)$ holds by the definitions of the operators $S^\ke(z)$ and of the transformation $\tau^{\ke_1,\ke_2}_\gamma$.
\end{itemize}

\subsection{Proof of Theorem \ref{equiv Thm2}}

For every fixed point $\mu\in (\WmodG)^{\T}$ put 
$$S_{1,\mu}:= \lan  S^\ke_{\bf t}(z)(P^\ke({\bf t}, z)) ,\delta_\mu\ran,$$
$$S_{2,\mu}:= \lan  S^\infty_{\tau^{\infty,\ke}({\bf t})}(z)(P^{\infty,\ke}(\tau^{\infty,\ke}({\bf t}),z))  ,\delta_\mu\ran.$$
Then $S_{j,\mu}$ are power series in $\Lambda[[t_i]]\{\!\{z,z^{-1}\}\!\}$. By their definition, they satisfy properties $(4)$ and $(5)$ in the Uniqueness Lemma. 
By Remark \ref{modification}, they satisfy property $(3)$ as well.

Recall that in Theorem \ref{equiv Thm2} we make the additional assumption that the $1$-dimensional $\T$-orbits in $\WmodG$ are isolated. This implies that the unbroken
components of the $\T$-fixed loci are of dimension zero. 

Let $o(\mu)$ denote the set of all fixed points $\nu$ connected to $\mu$ by a $1$-dimensional $\T$-orbit.
The proof of Lemma \ref{poles} shows that each $(q,\{t_i\})$-coefficient of $S_{j,\mu}$ is a rational function in $z$
which decomposes as the sum of a Laurent polynomial in $K[z,z^{-1}]$ and a sum of partial fractions with {\it simple} poles at $\frac{w(\mu,\nu)}{-n}$, where $n$ are positive integers and for all $\nu\in o(\mu)$, $w(\mu,\nu)$ is the tangent weight at $\mu$
of the corresponding orbit. Furthermore, $S_{j,\nu}$ is regular at $z=\frac{w(\mu,\nu)}{-n}$.

Finally, we claim that for $j=1,2$ the system $\{S_{j,\mu}\;|\; \mu\in (\WmodG)^{\T}\}$ satisfies Givental's recursion relation \eqref{isolated recursion}. Precisely,
$$S_{j,\mu}  (z) = R_{j,\mu} (z) + \sum _{\nu\in o(\mu)} \sum_{n=1}^{\infty}q^{n\beta(\mu,\nu)}\frac{C_{\mu,\nu,n}}{z+\frac{w(\mu,\nu)}{n}}S_{j,\nu}  (-\frac{w(\mu,\nu)}{n}),$$
with $R_{j,\mu} (z)$ having each $(q,\{t_i\})$-coefficient in $K[z,z^{-1}]$ and the recursion coefficients $C_{\mu,\nu,n}\in K$, not depending on $j$.
This is very easy to see. 

For example, take $j=1$. For each $\nu\in o(\mu)$ and each $n\geq 1$, write
\begin{equation}\label{pole removal}
P^\ke(z)=P^\ke(-\frac{w(\mu,\nu)}{n})+(z+\frac{w(\mu,\nu)}{n})A_{\nu,n}^\ke(z).
\end{equation}
For each $(m,d)$, the part of bi-degree $(m,d)$ of $A_{\nu,n}^\ke(z)$ is a polynomial in $z$.

Calculating $S_{1,\mu}(z)=i_\mu^*(S^\ke(z)(P^\ke(z))$ via virtual localization as in the proof of Lemma \ref{recursion lemma} gives
$$S_{1,\mu}(z)=\hat{R}^\ke_\mu(z)+\sum _{\nu\in o(\mu)} \sum_{n=1}^{\infty}q^{n\beta(\mu,\nu)}\frac{C_{\mu,\nu,n}}{z+\frac{w(\mu,\nu)}{n}} i_\nu^*
S^\ke(-\frac{w(\mu,\nu)}{n})(P^\ke(z)),$$
where $\hat{R}^\ke_\mu(z)$ is obtained by summing the contributions from the fixed-point components of initial type. As power series of $q$ and $t_i$,  
$\hat{R}^\ke_\mu(z)$ has coefficients in $K[z,z^{-1}]$. The recursion coefficients $C_{\mu,\nu,n}$ do not depend on $\ke$, or on the insertion $\gamma$ in $S(z)(\gamma)$.
Now use \eqref{pole removal} to get
\begin{align*}
S_{1,\mu}(z)&=
\hat{R}^\ke_\mu(z)+\sum _{\nu\in o(\mu)} \sum_{n=1}^{\infty}q^{n\beta(\mu,\nu)}C_{\mu,\nu,n} i_\nu^*S^\ke(-\frac{w(\mu,\nu)}{n})(A_{\nu,n}^\ke(z)) +\\
&+\sum _{\nu\in o(\mu)} \sum_{n=1}^{\infty}q^{n\beta(\mu,\nu)}\frac{C_{\mu,\nu,n}}{z+\frac{w(\mu,\nu)}{n}} i_\nu^*S^\ke(-\frac{w(\mu,\nu)}{n})(P^\ke(-\frac{w(\mu,\nu)}{n})).
\end{align*}
Setting
$$R_{1,\mu} (z)=\hat{R}^\ke_\mu(z)+\sum _{\nu\in o(\mu)} \sum_{n=1}^{\infty}q^{n\beta(\mu,\nu)}C_{\mu,\nu,n} i_\nu^*S^\ke(-\frac{w(\mu,\nu)}{n})(A_{\nu,n}^\ke(z)),$$
gives the claimed recursion relation for $S_{1,\mu}(z)$. The same argument obviously works for $S_{2,\mu}(z)$. 

Note that the proof of the Uniqueness Lemma goes through if we allow the initial terms in the recursion relation to have coefficients in $K[z,z^{-1}]$ rather than in $K[z^{-1}]$.
Hence we may apply it to the systems $\{S_{j,\mu}\;|\; \mu\in (\WmodG)^{\T}\}$ to conclude the Theorem.

\end{document}